\RequirePackage[reqno]{amsmath}
\documentclass[15pt,a4paper,oneside]{article}

\usepackage{environ}
\makeatletter
\NewEnviron{Ralign}{\tagsleft@false\begin{align}\BODY\end{align}}
\makeatother

\usepackage{epic,eepic}
\usepackage{amsmath}
\usepackage{amsfonts}
\usepackage{a4wide}
\usepackage{amssymb}
\usepackage{amsthm}
\usepackage{epsfig}
\usepackage{url}
\usepackage{bbm} 
\usepackage{dsfont}
\usepackage{yhmath}
\usepackage{palatino}
\usepackage{mathrsfs}
\usepackage{afterpage}
\usepackage{url}
\usepackage[T1]{fontenc}
\usepackage[utf8]{inputenc}
\usepackage{relsize}
\usepackage{braket}
\usepackage{stmaryrd}
\usepackage{enumitem}
\usepackage{multirow}
\usepackage{caption}
\usepackage[dvipsnames,svgnames,table,xcdraw]{xcolor}
\definecolor{citegreen}{rgb}{0,0.6,0}
\definecolor{refred}{rgb}{0.8,0,0}
\usepackage[colorlinks, citecolor=citegreen,linkcolor=refred]{hyperref}

\usepackage{mathtools}
\mathtoolsset{showonlyrefs,showmanualtags}

\mathchardef\emptyset="001F

\theoremstyle{plain}
\newtheorem{theorem}{Theorem}[section]
\newtheorem{lemma}[theorem]{Lemma}
\newtheorem{proposition}[theorem]{Proposition}
\newtheorem{corollary}[theorem]{Corollary}
\theoremstyle{definition}
\newtheorem{remark}{Remark}[section]
\newtheorem{definition}{Definition}[section]

\makeindex

\newcommand{\R}{\mathbb{R}}

\newcommand{\N}{\mathbb{N}}

\newcommand{\SSS}{\mathbb{S}}

\newcommand{\mres}{\mathbin{\vrule height 1.6ex depth 0pt width
0.13ex\vrule height 0.13ex depth 0pt width 1.3ex}}

\DeclareFontFamily{U}  {MnSymbolC}{}
\DeclareSymbolFont{MnSyC}         {U}  {MnSymbolC}{m}{n}
\SetSymbolFont{MnSyC}       {bold}{U}  {MnSymbolC}{b}{n}
\DeclareFontShape{U}{MnSymbolC}{m}{n}{
    <-6>  MnSymbolC5
   <6-7>  MnSymbolC6
   <7-8>  MnSymbolC7
   <8-9>  MnSymbolC8
   <9-10> MnSymbolC9
  <10-12> MnSymbolC10
  <12->   MnSymbolC12}{}
\DeclareFontShape{U}{MnSymbolC}{b}{n}{
    <-6>  MnSymbolC-Bold5
   <6-7>  MnSymbolC-Bold6
   <7-8>  MnSymbolC-Bold7
   <8-9>  MnSymbolC-Bold8
   <9-10> MnSymbolC-Bold9
  <10-12> MnSymbolC-Bold10
  <12->   MnSymbolC-Bold12}{}
\DeclareSymbolFont{MnSyC} {U} {MnSymbolC}{m}{n}
\DeclareMathSymbol{\utimes}{\mathrel}{MnSyC}{36}

\usepackage[a4paper]{geometry}
\usepackage[dvipsnames,svgnames,prologue,table]{pstricks}
\usepackage{graphicx}
\usepackage{pst-grad}
\usepackage{pst-slpe}
\usepackage{pst-text}
\usepackage{tikz}
\usetikzlibrary{matrix,tikzmark,calc,arrows,shapes,decorations.pathreplacing}
\tikzset{every picture/.style={remember picture}}

\usepackage{empheq}
\usepackage[most]{tcolorbox}
\newtcbox{\mymath}[1][]{nobeforeafter,math upper,tcbox raise base,enhanced,colframe=black,colback=white,boxrule=1pt,#1}

\usepackage{float}
\usepackage{enumitem}

\usepackage{palatino}
\usepackage{a4wide}

\begin{document}

\title{\textbf{\large{A GEOMETRIC CAPACITARY INEQUALITY FOR SUB-STATIC MANIFOLDS WITH HARMONIC POTENTIALS}}}
\date{}
%
%

\author{Virginia Agostiniani\\
\small{\em{Universit\`a degli Studi di Trento,
via Sommarive 14, 38123 Povo (TN), Italy}}\\
\small{\em{virginia.agostiniani@unitn.it}}
\and
Lorenzo Mazzieri\\
\small{\em{Universit\`a degli Studi di Trento,
via Sommarive 14, 38123 Povo (TN), Italy}}\\
\small{\em{lorenzo.mazzieri@unitn.it}}
\and
 Francesca Oronzio\\
\small{\em{Universit\`a di degli Studi di Napoli ``Federico II”, via Cintia, Monte S. Angelo 80126 Napoli (NA), Italy}}\\
\small{\em{francesca.oronzio@unina.it}}}

\maketitle

\begin{quote}
{\small ABSTRACT: In this paper, we prove that associated with a sub-static asymptotically flat manifold endowed with a harmonic potential there is a one-parameter family $\{F_{\beta}\}$ of functions which are monotone along the level-set flow of the potential. Such monotonicity holds up to the optimal threshold $\beta=\frac{n-2}{n-1}$ and allows us to prove a
geometric capacitary inequality where the capacity of the horizon plays
the same role as the ADM mass in the celebrated Riemannian Penrose Inequality.}
\end{quote}



\section{Introduction}
In this paper, the object under investigation is 
a triple $(M,g_{0},u)$ satisfying
the following two conditions:
\begin{enumerate}[ label=(\alph*)]
\item $(M,g_{0})$ is a smooth, connected, noncompact, complete, asymptotically flat, $n$-dimensional Riemannian manifold, with $n\geq 3$, with one end, and with nonempty smooth compact boundary $\partial M$, which is a priori allowed to have several connected components. \\
\item $u\in C^{\infty}(M)$ satisfies the system

\begin{equation}\label{f0}
\begin{cases}
u \mathrm{Ric}_{g_{0}}-\mathrm{D}_{g_{0}}^{2}u \geq 0 &\mathrm{in} \ M,\\
\Delta_{g_{0}} u=0 \  &\mathrm{in} \  M,\\
u=0 &\mathrm{on} \  \partial M,\\
u \to 1 &\mathrm{at} \  \infty,
\end{cases}
\end{equation}
where $\mathrm{Ric}_{g_{0}}$, $\mathrm{D}_{g_{0}}$ and $\Delta_{g_{0}}$ are the Ricci tensor, the Levi--Civita connection, and the Laplace operator of the metric $g_{0}$, respectively.
\end{enumerate}

\noindent 
If the equality holds in the first equation of \eqref{f0},
the triple $(M,g_0,u)$ is said \emph{static}.  
For clarity, we recall the definition to which we refer for asymptotically flat manifolds.
\begin{definition}\label{asymptoticallyflatmanifold}
A smooth, connected, noncompact, $n$-dimensional
Riemannian manifold (with or without compact boundary) $(N,h)$, with $n\geq 3$, is said to be {\em asymptotically flat} if there exists a compact subset $K\subset N$ such that $N \setminus K$ is a finite disjoint union of {\em ends} $N_{k}$,
each of which is diffeomorphic to $\R^{n}$ minus a closed ball by a coordinate chart $\psi_{k}$, through which, if $\widetilde{h}:=(\psi_{k})_{*}h=\widetilde{h}_{ij}dx^{i}\otimes dx^{j}$, we have
\begin{align}
\widetilde{h}_{ij}&= \delta_{ij} + O(\vert x \vert ^{-p})\,,\,\label{eq5}\\
\partial_{r}\,\widetilde{h}_{ij}&= O(\vert x \vert ^{-(p+1)}\,)\,,\,\label{eq6}\\
\partial_{r}\partial_{s}\,\widetilde{h}_{ij} & = O (\vert x \vert ^{-(p+2)}\,)\,,\,\label{eq7} \\
\mathrm{R}_{\widetilde{h}}&\in L^{1}(\mu_{\widetilde{h}})\,,\,\label{eq8}
\end{align}
for some $p>(n-2)/2$. Here, $\delta$ is the Kronecker delta, and the coordinate charts $\psi_{k}$ are called {\em charts at infinity}.
\end{definition}

\noindent Throughout the paper, we will refer to a triple $(M,g_{0},u)$ that satisfies conditions 
$(\mathrm{a})$ and $(\mathrm{b})$ as to a {\em sub-static harmonic triple}.
A fundamental sub-static harmonic triple is the so called {\em Schwarzschild solution}, which is given by
\begin{equation}\label{solschw}
M=[(2m)^{ \frac{1}{n-2}},+\infty)\times \SSS^{n-1}\,,\quad\,\,\,g_{0}=\frac{dr \otimes dr}{1-2mr^{2-n}}+r^{2}g_{\SSS^{n-1}}\,, \quad \,\,\, u=\sqrt{1-2m r^{2-n}}\,. 
\end{equation}
It is well--known that both the metric $g_{0}$ and the potential $u$, which a priori are well defined only in $\mathring{M}$, extend smoothly up to the boundary and $(M,g_{0})$ is called {\em (spatial) Schwarzschild manifold}.
The parameter $m > 0$ is the ADM mass $m_{\mathrm{ADM}}$ of the Schwarzschild manifold. We refer the reader to Section~\ref{sec_uniqueness} for the definition
of the $m_{\mathrm{ADM}}$ associated with a general asymptotically flat manifold.
Here, we limit ourselves to recall that the decay conditions~\eqref{eq5}--~\eqref{eq8} guarantee that $m_{\mathrm{ADM}}$ is a geometric invariant (\cite{Bartnik}, \cite{Chrusciel}).\phantom{~\eqref{eq6},~\eqref{eq7}}

\smallskip
\noindent
Associated with a sub-static harmonic triple, specifically
with the potential $u$ ranging in $[0,1)$,
let us consider the following family of functions
depending on the parameter $\beta\geq0$:
\[
[0,1)\ni\ t\longmapsto V_\beta(t)\,:=\,(1-t^2)^{-\beta(\frac{n-1}{n-2})}\!\!\!
\int\limits_{\{u=t\}}\!\vert\mathrm{D}u \vert^{\beta+1}\,d\sigma.
\]
In~\cite{Virginia1} it was proven that if $(M,g_0,u)$ is a static triple, then,
for every $\beta\geq2$, the function
$V_\beta$ is strictly nonincreasig unless $(M,g_0,u)$ is the Schwarzschild solution.
The main purpose of this paper is to extend this result to the sub-static case 
and to the optimal threshold $\beta=\frac{n-2}{n-1}$. 
This is the content of Theorem \ref{Monotonicity--Partial Rigidity Theorem},
where the monotonicity of the above family - equipped with a corresponding rigidity 
statement - is expressed in terms of the functions $F_\beta(\tau)$, where $\tau=\frac{1+t^2}{1-t^2}\geq1$,
to be consistent with~\cite{MON} and in light of the more advanced analysis
contained therein.
This generalisation suggests that our approach is robust enough and likely to be exported
to other contexts. 
In a similar way, S. Brendle shows in~\cite{ALEX} how some structure conditions for the metric
are sufficient to prove an Alexandrov-type theorem and how such structure
generalises to the sub-static case.

\smallskip
\noindent
Let us now be slightly more detailed on how our Theorem~\ref{Monotonicity--Partial Rigidity Theorem} 
is proved. We adopt the main strategy proposed in~\cite{Virginia1},
which essentially consists in obtaining the monotonicity as a consequence
of a fundamental integral identity 
derived in a suitable conformally-related setting
(see~Proposition \ref{itegralidenty}). 
A delicate point is justifying such identity in a region
where critical points of the potential are present. 
One of the main differences with~\cite{Virginia1} is that, whereas in the 
static case the analyticity of the potential 
guaranteed the local finiteness
of the singular values, which made the argument simpler
in many occurrences, in the present sub-static setting the metric and 
in turn the potential are not a priori analytic.
Nevertheless, standard measure properties of 
the critical set of harmonic functions (summarised in Theorem~\ref{geometryoflevelset})
are enough to obtain the fundamental integral 
identity,
which in turn implies the monotonicity of $F_\beta$ and, coupled with Sard's Theorem, 
also its differentiability.

\noindent
Observe that the difficulty in treating the critical points
under the threshold $\beta=1$ can be read off directly 
from formul\ae~\eqref{derivata_formale} and~\eqref{H}, the first one
displaying the derivative
of $F_\beta$ and the second one expressing
the mean curvature on a equipotential set
in terms of the Hessian of the potential itself.
In fact, calling $\Phi_\beta$ the conformal version of $F_\beta$
and looking at formula~\eqref{Phi'} containing the equivalent characterisation 
of $\Phi_\beta'$ derived from the integral identity~\eqref{partialfirstintegralidentities2},
one realises that problems arise already when $\beta<2.$

\noindent
Let us stress that the monotonicity is
obtained from the nonnegativity
of the right-hand side of our
fundamental integral identity. It is above the threshold $\beta=\frac{n-2}{n-1}$
that this is guaranteed, thanks to the Refined Kato Inequality
for harmonic functions. The optimality of such inequality
reflects a corresponding optimality of $\beta=\frac{n-2}{n-1}$ in our
result.
Moreover, let us remark that the (nonnegative) right-hand side of
\eqref{partialfirstintegralidentities2} is obtained
as the divergence a suitable modification
of a specific vector filed with nonnegative divergence (see \eqref{divY}), in the limit of a vanishing neighbourhood of the critical set.
The crucial point in the construction is to maintain the divergence
of the modified vector field  nonnegative.
It would be interesting to see whether a similar construction
can be performed for other families of metrics, including
special solutions as rigid case.

\smallskip
\noindent 
A straightforward application of the monotonicity of $F_\beta$ is comparing
$F_\beta(1)$ with $F_\beta(+\infty)$, in turn yielding a ``capacitary version''
of the Riemannian-Penrose inequality
(Theorem~\ref{capacitaryriemannianpenroseinequality} below). 
The capacity comes naturally into play when computing 
$F_\beta(1)$ and $F_\beta(+\infty)$, the latter value via the asymptotic 
expansions of the metric and of the potential.
We recall that the capacity 
$\mathrm{Cap}(\partial M,g_{0})$ of $\partial M$ is defined as
\begin{align}\label{capacity}
\mathrm{Cap}(\partial M,g_{0}):=\frac{1}{(n-2)\vert \SSS^{n-1}\vert}\inf 
\Bigg\{ \int\limits_{M}\vert  \mathrm{D}_{g_{0}} v \vert _{g_{0}}^{2}\,d\mu_{g_{0}}:\, v\in \mathrm{Lip}_{loc}(M), \, v=0\,\, \text{on}\,\, \partial M,\, v\to 1\,\text{at}\, \infty\Bigg\}\,.
\end{align}
Throughout the paper, we will use the short--hand notation
$\mathcal C$ for the capacity. Comparing~\eqref{solschw} 
with either~\eqref{eq9} or~\eqref{eq10}, 
it is straightforward, in the case of
the Schwarzschild solution, that $m_{\mathrm{ADM}}=\mathcal C$.
For a general sub-static harmonic triple, the following inequality holds.

\begin{theorem}[Capacitary Riemannian Penrose Inequality]\label{capacitaryriemannianpenroseinequality}
Let $(M,g_{0},u)$ be a sub-static harmonic triple with associated capacity $\mathcal{C}$ and suppose that $\partial M$ is connected. Then 
\begin{align}\label{eq15}
\mathcal{C}\geq \frac{1}{2}\Bigg(\frac{\vert \partial M\vert\,\,}{\,\,\vert \SSS^{n-1}\vert}\Bigg)^{\!\!\frac{n-2}{n-1}}\,.
\end{align}
Moreover, the equality in~\eqref{eq15} holds if and only if $(M,g_{0})$ is isometric to the Schwarzschild manifold with $m_{\mathrm{ADM}}=\mathcal{C}$.
\end{theorem}

\noindent Whereas the above inequality has been obtained as a consequence of the
monotonicity of $F_\beta$, at every fixed $\beta\geq\frac{n-2}{n-1}$,
we remark that one could possible push the above described analysis 
one step forward, at the same time exploiting the full power of the optimality threshold. Indeed,
we believe that considering $p$-harmonic functions defined at the exterior 
of a bounded domain $\Omega$ lying in $M$, 
it may be possible to derive, 
as done in~\cite{MINK} for the Euclidean case and in the simultaneous limit as 
$\beta\downarrow\frac{n-2}{n-1}$ and $p\downarrow1$, 
a Minkowski-like inequality for $\partial\Omega$ 
(see~\cite{McC} for a Minkowski-like inequality in the static, 
asymptotically flat case).

\noindent
Concerning the treatment of general sub-static metrics and the derivation of
related geometric inequalities, besides the already cited~\cite{ALEX} 
we also would like to mention
~\cite{LI-XIA}, where an integral formula is obtained and applied to prove 
Hentze-Karcher-type inequalities.
For the case of asymptotically hyperbolic sub-static manifolds 
(specifically, for adS-Reissner-Nordstr\"om manifolds), we refer the interested reader to 
~\cite{WANG} and~\cite{GIR-ROD}.

\smallskip
\noindent We remark that our results are not based on the Positive Mass Theorem.
By contrast, we observe that using this celebrated result, more precisely a consequence
of it contained in~\cite[Theorem 1.5]{hirsch}, one can prove the following uniqueness statement. We refer the reader to Definition~\ref{asymptoticallyflatmanifold}
for the notation and terminology.

\begin{theorem}[Uniqueness Theorem for sub-static harmonic triples]
\label{uniq}
Let $(M,g_{0},u)$ be a sub-static harmonic triple with associated capacity $\mathcal{C}$. Suppose that there is a chart at infinity such that 
\begin{align}\label{eq58}
\mathrm{R}_{\widetilde g_0}=O(\vert x\vert ^{-q})\,,
\end{align}
for some $q>n$. Then $(M,g_{0})$ is the Schwarzschild manifold with associated ADM mass given by $ \mathcal{C}$.
\end{theorem}
\noindent
It remains an open question to see whenever it is possible to remove the assumption on the decay of $\mathrm{R}_{\widetilde{g}_{0}}$ and get the same conclusion.

\bigskip
\noindent The paper is organised as follows. 
In Section~\ref{prel}, we recall and discuss some preparatory material, namely the asymptotic 
expansions of the metric and of the potential, and classical measure properties of the critical set of the potential, with a close look on related integral quantities. 
In Section~\ref{MonOutRigTheor}, we prove the Monotonicity and Outer Rigidity Theorem~\ref{Monotonicity--Partial Rigidity Theorem}, and the consequent Capacitary Riemannian Penrose Inequality contained in Theorem~\ref{capacitaryriemannianpenroseinequality}. To do this, we use from
Section~\ref{conformalsetting} some corresponding results obtained in a suitable conformally-related
setting. The biggest technical effort is contained in such section.
In the Appendix we also provide an alternative proof of the monotonicity of
our monotone quantities.
Finally, Section~\ref{sec_uniqueness} is devoted to the proof of Theorem~\ref{uniq}.


\section{Preliminaries}\label{prel}
\noindent
Let $(M,g_{0},u)$ be a sub-static harmonic triple.
We observe, as a first consequence of system~\eqref{f0}, that the scalar curvature $\mathrm{R}_{g_{0}}$ is nonnegative.
Since $u$ satisfies the last three conditions of system~\eqref{f0}, by the Maximum Principle we have 
$$\mathring{M}=M\setminus \partial M=\{0<u<1\}\,.$$
Also, by the forth condition in~\eqref{f0}, each level set of $u$ is compact.
Moreover, from the Hopf Lemma, it follows that 
$\vert \mathrm{D}_{g_{0}}u \vert _{g_{0}}>0$ on $\partial M$. In particular, zero is a regular value of $u$.
Furthermore, from the first two conditions in~\eqref{f0} restricted to $\partial M$ it is easy to deduce that $\mathrm{D}_{g_{0}}^{2}u\equiv 0$ on $\partial M$. In turn, the function $\vert \mathrm{D}_{g_{0}}u \vert_{g_{0}}$ attains a positive constant value on each connected component of $\partial M$, and the boundary $\partial M$ is a totally geodesic hypersurface in $M$.
\medskip

\noindent
We now deal with the asymptotic behaviour of the potential $u$ at $\infty$.
By Theorem~\ref{asympespanprel} below, this is given by:
\begin{equation}\label{eq9}
u=1-\frac{\mathcal{C}}{\vert x\vert ^{n-2}}+o_{2}(\vert x\vert^{2-n})\ \ \,\text{as} \ \ \vert x\vert\to +\infty\,,
\end{equation}
being
\begin{equation}\label{eq10}
\mathcal{C}=\frac{1}{(n-2)\vert \SSS^{n-1}\vert}\int\limits_{\partial M}\vert \mathrm{D}_{g_{0}} u \vert _{g_{0}}d\sigma_{g_{0}}\,.
\end{equation}
Here, $\sigma_{g_{0}}$ is the canonical measure on the boundary $\partial M$ seen as a Riemannian submanifold of $(M,g_{0})$, and we have used the standard notation $o_{2}$, which means that, in any chart at infinity $\psi$, denoted by $\widetilde{u}$ the function $u\circ \psi^{-1}$, the following conditions hold true.
\begin{align}
\widetilde{u}&=1-\mathcal{C}\vert x\vert ^{2-n}\,+\,o(\vert x\vert^{2-n})\,, \label{eq11}\\
\partial_{i}\widetilde{u}&=(n-2)\,\mathcal{C}\,\vert x\vert^{-n}\,x^{i}\,+\,o(\vert x\vert^{1-n})\,,\label{eq12}\\
\partial_{i}\partial_{j}\widetilde{u}&=-(n-2)\,\mathcal{C}\,\vert x\vert^{-n-2}(n\,x^{i}\,x^{j}-\vert x\vert^{2}\delta_{ij})\,+\,o(\vert x\vert^{-n})\,.\label{eq13}
\end{align}
Let us remark that we can always suppose, without loss of generality, 
that the considered chart at infinity admits a diffeomorphic extension 
to the closure of the coordinate domain. We will make this implicit assumption
throughout the paper, so that $\partial K$ 
(see Definition \ref{asymptoticallyflatmanifold}) is a connected hypersurface of $M$ and the quantities related to the metric can be pushed--forward in $\R^{n}$ outside an open ball and be smooth here.
We also observe that formula \eqref{eq10} is nothing but an equivalent
characterisation of the capacity of $\partial M$.

\subsection{Asymptotic expansions}\label{appA}
Let $(N,h)$ be a smooth, connected, noncompact, complete, asymptotically flat, $n$-dimensional Riemannian manifold, with $n\geq 3$, with one end and with nonempty smooth compact boundary $\partial N$.
We adopt the following notation. 
\begin{itemize}
\item $B$ and $B_{R}$ a generic open ball and the open ball of radius $R>0$ centred in the origin of 
$(\R^{n},d_{e})$, respectively;
\item $\vert \cdot \vert$ the euclidean norm of $\R^{n}$;
\item $\vert \SSS^{n-1}\vert$ the hypersurface area of the unit sphere inside $\R^{n}$ with the canonical metric;
\item $\mathrm{D}_{e}$ and $\Delta_{e}$ the Levi--Civita connection and the Laplace operator of $ (\R^{n}, g_{\R^{n}})$, respectively;
\item $\mathrm{D}_{h}$ and $\Delta_{h}$ the Levi--Civita connection and the Laplace operator of $(N,h)$, respectively;
\item $\sigma_{e}$ the canonical measure on a Riemannian submanifold of $(\R^{n},g_{\R^{n}})$;
\item $\sigma_{h}$ the canonical measure on a  Riemannian submanifold of $(N,h)$;
\item $\vert \cdot \vert_{e}$ the norm induced by $g_{\R^{n}}$ on the tangent spaces to the manifold $\R^{n}$;
\item $\vert \cdot \vert_{h}$ the norm induced by $h$ on the tangent spaces to the manifold $N$.
\item If $\psi$ is a chart at infinity of $(N,h)$ according to Definition~\ref{asymptoticallyflatmanifold}, we denote by $\widetilde{h}$ the push--forward metric $\psi_{*}{h}$ of $h$ by $\psi$, having coordinate expression $\widetilde{h}_{ij}(x)\,dx^{i}\otimes dx^{j}$. In this context, $\mathrm{D}_{\widetilde{h}}$ and $\Delta_{\widetilde{h}}$ denote the Levi--Civita connection and the Laplace operator of $\widetilde{h}$, respectively, while $\sigma_{\widetilde{h}}$ is the canonical measure on a Riemannian submanifold of $(\R^{n}\setminus B,\widetilde{h})$ 
and $\vert \cdot \vert_{\widetilde{h}}$ is the norm induced by $\widetilde{h}$ on the tangent spaces. Moreover, $\mathrm{Ric}_{\widetilde{h}}$ and $\mathrm{R}_{\widetilde{h}}$ are the Ricci tensor and the scalar curvature of $\widetilde{h}$, respectively.
\end{itemize}

\begin{proposition}
Let $\psi$ be a chart at infinity of $N$ (according to Definition~\ref{asymptoticallyflatmanifold}). The decays
\begin{align}
\widetilde{h}^{ij}-\delta^{ij}&=O_{2}(\vert x \vert ^{-p})\,,\,\label{eq37app}\\
(\mathrm{Rm}_{\widetilde{h}})_{ijk}^{l}&=O(\vert x \vert ^{-(p+2)})\,,\,\label{eq38app}\\
(\mathrm{Ric}_{\widetilde{h}})_{ij}&=O(\vert x \vert ^{-(p+2)})\,,\,\label{eq39app}\\
\mathrm{R}_{\widetilde{h}}&=O(\vert x \vert ^{-(p+2)})\,,\,\label{eq40app}
\end{align}
hold true for some $p>\frac{n-2}{2}$. Moreover, 
\begin{align}
\vert \nu_{\,\widetilde{h}}^{i}-\nu_{e}^{i}\vert&=O(\vert x \vert ^{-p})\,\label{eq42app}\,,\\
d\sigma_{\widetilde{h}}&=(1+O(\vert x \vert ^{-p})\,d\sigma_{e}\,\label{eq41app}\,,
\end{align}
where $\nu_{e}$ is the $\infty$--pointing unit normal with respect to the Euclidean metric and $\sigma_{e}$ the associated canonical measure on $\partial B_{R}\,$, while $\nu_{\,\widetilde{h}}$ is the $\infty$--pointing unit normal with respect to $\widetilde{h}$ and $\sigma_{\widetilde{h}}$ the associated canonical measure on $\partial B_{R}\,$.
\end{proposition}
\begin{proof}
From $\widetilde{h}^{ik} \widetilde{h}_{kj} =\delta^{i}_{j}$ it is easy to get
\begin{align}
\partial_{i}\widetilde{h}^{kl}&=-\widetilde{h}^{kr}\, \widetilde{h}^{ls}\,\partial_{i}\widetilde{h}_{rs}\\
\partial_{i}\partial_{j}\widetilde{h}^{kl}&=\widetilde{h}^{ka}\, \widetilde{h}^{rb}\, \widetilde{h}^{ls}\,(\partial_{i}\widetilde{h}_{ab})\,(\partial_{j}\widetilde{h}_{rs})+\widetilde{h}^{la}\, \widetilde{h}^{sb}\, \widetilde{h}^{kr}\,(\partial_{i}\widetilde{h}_{ab})\,(\partial_{j}\widetilde{h}_{rs})-\widetilde{h}^{kr}\, \widetilde{h}^{ls}\,\partial_{i}\partial_{j}\widetilde{h}_{rs}\,.
\end{align}
These formulae coupled with~\eqref{eq5},~\eqref{eq6} and~\eqref{eq7} give~\eqref{eq37app}.
Decay~\eqref{eq38app} is another direct consequence of Definition~\ref{asymptoticallyflatmanifold}, keeping in mind that
\begin{align}
(\mathrm{Rm}_{\widetilde{h}})_{ijk}^{l}&=\partial_{i}\Gamma_{jk}^{l}-\partial_{j}\Gamma_{ik}^{l}+\Gamma_{jk}^{s}\Gamma_{is}^{l}-\Gamma_{ik}^{s}\Gamma_{js}^{l}\,,\\
\Gamma_{ij}^{k}&=\frac{\widetilde{h}^{kl}}{2}[\partial_{i}\widetilde{h}_{lj}+\partial_{j}\widetilde{h}_{li}-\partial_{l}\widetilde{h}_{ij}]\,.\label{simbchristoffel}
\end{align}
Decays~\eqref{eq39app}--\eqref{eq40app} are obtained by contractions of the Riemannian tensor.
Now, observe that 
$$\nu_{\,\widetilde{h}}=\frac{\widetilde{h}^{ij}\,x^{i}\,\frac{\partial}{\partial x^{j}}}{\sqrt{\,\widetilde{h}^{lk}\,x^{l}\,x^{k}}}\,,$$
and that
\begin{align}
\vert \nu_{e}^{i}-\nu_{\,\widetilde{h}}^{i}\vert&=\Bigg\vert \,\frac{x^{i}}{\vert x\vert}-\frac{\widetilde{h}^{ij}\,x^{j}}{\sqrt{\,\widetilde{h}^{lk}x^{l}x^{k}} }\Bigg\vert =\Bigg\vert (\delta^{ij}-\widetilde{h}^{ij})\,\frac{x^{j}}{\sqrt{\,\widetilde{h}^{lk} x^{l}x^{k}}}\,+
 x^{i}\,\Bigg(\frac{1}{\vert x\vert} -\frac{1}{\sqrt{\,\widetilde{h}^{lk} x^{l}x^{k}}}\Bigg)\,\Bigg\vert\\
&\leq  C\sum\limits_{j}\vert   \delta^{ij}-\widetilde{h}^{ij}\vert\,+\,\frac{\Big\vert\, \big(\widetilde{h}^{lk}-\delta^{lk}\big) \,x^{l}x^{k}\,\Big\vert}{\sqrt{\,\widetilde{h}^{lk}x^{l}x^{k}}\Big(\sqrt{\,\widetilde{h}^{lk} x^{l}x^{k}}+\vert x\vert\Big)}\,\label{eq250} \,.
\end{align}
Observe also that 
\begin{equation}\label{eq25}
\widetilde{h} _{ij}(x)v^{i}v^{j}\geq C ^{-1} v^{i}v^{j}\delta _{ij}\,,
\end{equation}
for some $C>0$, for any $x\in \R^{n}\setminus B$. 
Since trivially $\vert x^{k} x^{l}\vert \leq \vert x\vert^{2}$, from~\eqref{eq250} and~\eqref{eq25}, coupled with~\eqref{eq37app}, we get decay~\eqref{eq42app}.
Concerning decay~\eqref{eq41app}, recall first that, using a coordinate chart  $(y^{1},\dots,y^{n-1})$ on $\partial B_{R}$, we have that $d\sigma_{\widetilde{h}}=\sqrt{\det\,\widetilde{h}^{\partial B_{R}}}\,dy^{1}\dots dy^{n-1}$ with 
$\widetilde{h}^{\partial B_{R}}=\widetilde{h}^{\partial B_{R}}_{\alpha\beta}\,dy^{\alpha}\otimes dy^{\beta}$, where $\widetilde{h}^{\partial B_{R}}_{\alpha\beta}=\widetilde{h}\big(\frac{\partial}{\partial y^{\alpha}},\frac{\partial}{\partial y^{\beta}}\big)$.
Now, using the specific local parametrization $x=x(y^{1},\dots,y^{n-1})$ of $\partial B_{R}$, given by the inverse of stereographic projection from its north pole with the diffeomorphism $p\in\SSS^{n-1}\to Rp\in\partial B_{R}$, we have that
\begin{align*}
\widetilde{h}\big(\frac{\partial}{\partial y^{\alpha}},\frac{\partial}{\partial y^{\beta}}\big)(x(y))&=\widetilde{h}_{ij}(x(y))\,\frac{\partial x^{i}}{\partial y^{\alpha}}(y)\,\frac{\partial x^{j}}{\partial y^{\beta}}(y)\\
&=\big(\widetilde{h}_{ij}(x(y))\pm \delta_{ij}\big)\,\frac{\partial x^{i}}{\partial y^{\alpha}}(y)\,\frac{\partial x^{j}}{\partial y^{\beta}}(y)=\frac{4 R^{2}}{(\vert y\vert^{2}+1)^{2}}\,\Big(\delta_{\alpha\beta}+O(R^{-p})\Big)\,,
\end{align*}
because
\begin{equation}
\Big\vert \,\big(\widetilde{h}_{ij}(x(y))-\delta_{ij}\big)\,\frac{\partial x^{i}}{\partial y^{\alpha}}(y)\,\frac{\partial x^{j}}{\partial y^{\beta}}(y)\,\Big\vert
\leq \sum\limits_{i,j}\big\vert\widetilde{h}_{ij}(x(y))-\delta_{ij}\big\vert \, \Big\vert\frac{\partial}{\partial y^{\alpha}}\Big\vert_{e}\,  \,\Big\vert\frac{\partial }{\partial y^{\beta}}\Big\vert_{e}= \frac{4 R^{2}}{(\vert y\vert^{2}+1)^{2}}\, O(R^{-p})\,.
\end{equation}
Hence, on $\partial B_{R}$,
$$d\sigma_{\widetilde{h}}=\Big(\frac{2 R}{\vert y\vert^{2}+1}\Big)^{n-1}\,\sqrt{\det\big(\delta_{\alpha\beta}+O(R^{-p})\big)}\,dy^{1}\dots dy^{n-1}=\big(1+O(R^{-p})\big)d\sigma_{e}\,,$$
where in the last identity we have used the Leibniz formula for the determinant and Taylor--expanded the square root.
\end{proof}

\noindent
The following result is well--known. For completeness, we provide the statement, along with its proof, which is an extension of~\cite[Lemma A.2.]{MMT} to every $n\geq3$.
\begin{theorem}\label{asympespanprel}
Let $(N,h)$ be a smooth, connected, noncompact, complete, asymptotically flat, $n$-dimensional Riemannian manifold, with $n\geq 3$, with one end, and with nonempty smooth compact boundary $\partial N$. If $v\in C^{\infty}(N)$ is the solution to
\begin{equation}\label{f1prel}
\begin{cases}
\Delta_{h} v=0 \  &\mathrm{in} \  N\,,\\
v=1 &\mathrm{on} \  \partial N\,,\\
v\to 0&\mathrm{at} \  \infty\,.
\end{cases}
\end{equation}
then
\begin{equation}\label{eq28prel}
v=\frac{\mathcal{C}}{\vert x\vert ^{n-2}}+o_{2}(\vert x\vert^{2-n})\ \ \text{as} \ \ \vert x\vert\to \infty\,, \quad \quad \text{with}\quad \quad \mathcal{C}=\frac{1}{(n-2)\vert \SSS^{n-1}\vert}\int\limits_{\partial N}\vert \,\mathrm{D}_{h} v\,\vert_{h}\, d\sigma_{h}\,.
\end{equation}
\end{theorem}

\noindent
We remark that the asymptotic behaviour of the potential $u$ at $\infty$, given by formula~\eqref{eq9}, is a simply consequence of the above theorem observing that $u=1-v$ when $(N,h)=(M,g_{0})$.
\medskip

\begin{proof}
{\em \underline{Step $1$: Construction of a barrier function.}}
Let $\psi$ be a chart at infinity for $N$. From now on by $C$ we will denote some positive constant, which may change from line to line.
By Definition~\ref{asymptoticallyflatmanifold}, there exist $p>(n-2)/2$ and $R_{1}\geq1$ such that
\begin{equation}
\R^{n}\setminus {B}_{R_{1}}\subseteq\R^{n}\setminus \overline{B}
\end{equation}
\begin{equation}\label{eq26}
\vert \widetilde{h}_{ij}- \delta_{ij}\vert \leq C\vert x \vert ^{-p}\quad\quad\quad \vert \partial_{k} \widetilde{h}_{ij}\vert \ \leq  C \vert x \vert ^{-(p+1)}\quad\quad\quad \vert\partial_{k}\partial_{l} \widetilde{h}_{ij}\vert \  \leq C \vert x \vert ^{-(p+2)}
\end{equation}
for every $x\in  \R^{n}\setminus B_{R_{1}}$.
By~\eqref{eq37app}, the same conditions as in~\eqref{eq26} are satisfied by $\widetilde{h}^{-1}(x)$ for all $x\in  \R^{n}\setminus B_{R_{1}}$.
Then, for every $f\in C^{\infty}( \R^{n}\setminus B_{R_{1}})$, writing
\begin{align}
\Delta_ {\widetilde{h}}f
&=\delta^{ij}\partial_{i}\partial _{j}f+\sigma^{ij}\partial_{i}\partial _{j}f+b^{j} \partial_{j}f\,\,,\label{eq27}
\end{align}
where 
$$\sigma^{ij}:=\widetilde{h}^{ij}-\delta^{ij}\,, \quad \quad b^{j}:=-\widetilde{h}^{kl}\Gamma_{kl}^{j}=\frac{1}{2}\, \widetilde{h}^{kl}\,\widetilde{h}^{ij}\, \partial_{i}\widetilde{h}_{lk}- \, \widetilde{h}^{ki}\,\widetilde{h}^{lj}\,\widetilde{h}_{kl,i}\,,$$ 
we have that
\begin{equation}\label{decdiffeq}
\vert\sigma^{ij}\vert \leq C\vert x \vert ^{-p}\,,\quad \quad\quad \vert b^{j}\vert\leq C \vert x \vert ^{-(p+1)}\,,\quad \quad\quad \vert \partial_{l}b^{j}\vert\leq C \vert x \vert ^{-(p+2)}
\end{equation}
in $ \R^{n}\setminus B_{R_{1}}$. For a fixed $0<\varepsilon<p$ and for $a>0$ to be chosen later, consider the function
$$\phi_{a}=a\Bigg(\frac{1}{\vert x \vert^{n-2}}-\frac{1}{\vert x \vert^{n-2+\varepsilon}}\Bigg)\,.$$
By direct computation one can check that
\begin{align}
\partial_{j}\phi_{a}&=-a\Bigg(\frac{n-2}{\vert x \vert^{n}}-\frac{n-2+\varepsilon}{\vert x \vert^{n+\varepsilon}}\Bigg)x^{j}\\
\partial_{i}\partial_{j}\phi_{a}&=a\Bigg[\frac{n\,(n-2)}{\vert x \vert^{n+2}}-\frac{ (n+\varepsilon)(n-2+\varepsilon)}{\vert x \vert^{n+2+\varepsilon}}\Bigg]x^{i}\,x^{j}-a\Bigg[\frac{n-2}{\vert x \vert^{n}}-\frac{n-2+\varepsilon}{\vert x \vert^{n+\varepsilon}}\Bigg]\delta_{ij}\,,
\end{align} 
and in turn that
\begin{equation}
\vert\partial_{i}\phi_{a}\vert \leq a\, C\vert x \vert^{1-n}\quad\quad \text{and}\quad\quad \vert\partial_{i}\partial_{j}\phi_{a}\vert \leq a \,C\vert x \vert^{-n}\,.
\end{equation}
Therefore, by~\eqref{eq27} and~\eqref{decdiffeq}, we obtain that
\begin{align}
\Delta_{\widetilde{h}} \phi_{a}=a\Big[-(n-2+\varepsilon)\varepsilon\vert x \vert^{-(n+\varepsilon)}+O(\vert x \vert^{-(n+p)})\Big]\,,
\end{align}
 and hence there exists $R_{2}> R_{1}$ independent of $a$ such that $\Delta_{\widetilde{h}} \phi_{a}<0$ in $\R^{n}\setminus B_{R_{2}}$, for every $a>0$.
We now choose $a>0$ so that $\phi_{a}=1$ on $\partial B_{R_{2}}$, that is $a=\Big[\frac{1}{R_{2}^{n-2}}-\frac{1}{R_{2}^{n-2+\varepsilon}}\Big]^{-1}$.
Since $\phi_{a}$ is $\widetilde{h}$--superharmonic in $\R^{n}\setminus B_{R_{2}}$ and since $\widetilde{v}:=v\circ \psi^{-1}<1$  on $\partial B_{R_{2}}$,
by the Maximum Principle 
\begin{align}\label{eq66}
\widetilde{v}\leq \phi_{a}\quad\text{in}\quad \R^{n}\setminus B_{R_{2}}\,.
\end{align}

\noindent
{\em \underline{Step $2$: Asymptotic expansion of $v$.}}
Note that from~\eqref{eq66} one gets in particular that $\widetilde{v}\leq C\vert x \vert^{2-n}$.
We now apply Shauder's Interior estimates (\cite[Lemma 6.20]{trudinger}) to $\Delta_{\widetilde{h}}\widetilde{v}=0$ in $\R^{n}\setminus \overline{B}_{R_{2}}$, where the operator $\Delta_{\widetilde{h}}$ is defined as in~\eqref{eq27} and its coefficients satisfy the estimates in~\eqref{decdiffeq}. Recalling that the H\"{o}lder norms are weighted by the (Euclidean) distance $d_{e}(\,\cdot\,,\partial B_{R_{2}})$ from $\partial B_{R_{2}}$ and since $d_{e}(x,\partial B_{R_{2}})\simeq\vert x\vert$ when $\vert x\vert>>1$, from such estimates we get
\begin{align}\label{decsol}
\vert\partial_{i}\widetilde{v}(x)\vert\leq C \vert x\vert^{1-n}\quad\quad\quad \vert\partial_{i}\partial_{j}\widetilde{v}(x)\vert\leq C \vert x\vert^{-n}
\end{align}
in $\R^{n}\setminus B_{R_{2}}$ (up to a bigger $R_{2}$). Combining~\eqref{decdiffeq} and~\eqref{decsol}, the equation $\Delta_{\widetilde{h}}\widetilde{v}=0$ can be equivalent written as $\Delta_{e}\widetilde{v} = f\,$ where
\begin{align}\label{neweucleq}
\vert f(x)\vert &\leq C \vert x\vert^{-(n+p)}\,.
\end{align}
We consider a smooth extension of $\widetilde{v}$ on $\R^{n}$, still denoted by $\widetilde{v}$, and the smooth extension of $f$ given by $\Delta_{e}\widetilde{v}$, still denoted by $f$. 
By a classical representation formula and due to~\eqref{neweucleq}, the function 
$$w(x)=-\frac{1}{n (n-2) \omega_{n}}\int\limits_{\R^{n}}\frac{f(y)}{\vert x-y \vert^{n-2}}\,dy\,,$$
is well--defined and fulfils $\Delta_{e} w=f$ on $\R^{n}$.
Now, one can rewrite $w$ in $\R^{n}\setminus \{O\}$ as 
\begin{align}
w(x)&=-\frac{1}{n (n-2) \omega_{n}}\frac{1}{\vert x\vert^{n-2}}\int\limits_{\R^{n}}f(y) \,dy+\frac{1}{n (n-2) \omega_{n}}\frac{1}{\vert x\vert^{n-2}}\int\limits_{\R^{n}\setminus B_{\frac{\vert x \vert}{2}}(O)}f(y)\, dy\\
&\quad-\frac{1}{n (n-2)\omega_{n}}\int\limits_{B_{\frac{\vert x \vert}{2}}(x)}\frac{f(y)}{\vert x-y \vert^{n-2}}\,dy-\frac{1}{n (n-2)\omega_{n}}\int\limits_{\R^{n}\setminus \big(B_{\frac{\vert x \vert}{2}}(x)\cup B_{\frac{\vert x \vert}{2}}(O)\big) }\frac{f(y)}{\vert x-y \vert^{n-2}}\,dy\\
&\quad-\frac{1}{n (n-2)\omega_{n}}\int\limits_{B_{\frac{\vert x \vert}{2}}(O)}\Bigg[\frac{1}{\vert x-y \vert^{n-2}}-\frac{1}{\vert x\vert^{n-2}}\Bigg]\,f(y)\,dy\,,
\end{align}
and show that each summand can be bounded by $C\vert x\vert^{-(n-2+p)}$, 
except the first one.
Therefore, we have that
\begin{equation}
w(x)=-\frac{1}{n (n-2) \omega_{n}}\frac{1}{\vert x\vert^{n-2}}\int\limits_{\R^{n}}f(y) \,dy+z(x)\,,\quad\quad\vert z(x)\vert \leq C\vert x \vert^{-(n-2+p)}\,,
\end{equation}
in $\R^{n}\setminus \{O\}$. 
Since the function $\widetilde{v}-w$ is harmonic and bounded on $\R^{n}$, then it is constant and this constant is zero, using the fact that $\widetilde{v}-w \to 0$ for $\vert x\vert \to \infty$.
Hence 
\begin{equation}\label{eq31}
\widetilde{v}=\frac{\mathcal{C}}{\vert x\vert^{n-2}}+z(x)
\end{equation}
in $\R^{n}\setminus B_{R_{2}}$.
We observe that
\begin{align}
\Delta_{\widetilde{h}}z=\Delta_{\widetilde{h}}\Bigg(\widetilde{v}-\frac{\mathcal{C}}{\vert x\vert^{n-2}}\Bigg)&=-\mathcal{C}\Big(\sigma^{ij}\partial_{i}\partial_{j}\frac{1}{\vert x\vert^{n-2}}+b^{k} \partial_{k} \frac{1}{\vert x\vert^{n-2}} \Big)=:-\mathcal{C}\hat{z}\,,
\end{align}
and that
\begin{equation}
\vert \hat{z}(x)\vert \leq C \vert x\vert^{-(n+p)}\,,\quad \quad\quad \vert \partial_{k}\hat{z}(x)\vert \leq C \vert x\vert^{-(n+p+1)}\,.
\end{equation}
Therefore, applying Shauder's Interior estimates to $\Delta_{\widetilde{h}}z= -\mathcal{C}\hat{z}$ in $\R^{n}\setminus \overline{B}_{R_{2}}$, we get
\begin{equation}\label{eq32}
\vert z(x)\vert \leq C \vert x\vert^{-(n-2+p)}\quad\quad\quad \vert\partial_{i} z(x)\vert\leq C \vert x\vert^{-(n-1+p)}\quad\quad\quad \vert\partial_{i}\partial_{j}z(x)\vert\leq C \vert x\vert^{-(n+p)}
\end{equation}
in $\R^{n}\setminus B_{R_{2}}$ (up to a bigger $R_{2}$). From~\eqref{eq31} and~\eqref{eq32} we obtain in particular~\eqref{eq28prel}.
\bigskip

\noindent 
{\em \underline{Step $3$: Characterization of $\mathcal{C}$.}}
First of all we remark that $0< v< 1$ on $\mathring{N}$, $v:N\to(0,1]$ is proper, and, from the Hopf Lemma,
$\vert \mathrm{D}_{h}v \vert _{h}>0$ on $\partial N$. In particular, $1$ is a regular value of $v$.
Let $K$ be the compact set on the complement of which the chart $\psi$  is defined.
For every $R>R_{2}$, applying the Divergence Theorem to the function $v$ on $K\cup \{\vert \psi\vert< R\}$ we obtain that
\begin{align}
0=\int\limits_{K\cup \{\vert \psi\vert< R\}}\Delta_{h} v \,d\mu_{h}=\int\limits_{\partial N}h(\mathrm{D}_{h}v,\nu_{h})\,d\sigma_{h}+\int\limits_{\{\vert \psi\vert =R\}}h(\mathrm{D}_{h}v, \nu_{h})\, d\sigma_{h}\,,
\end{align}
where $\nu_{h}$ is the outward unit normal vector field with respect to $h$ along $\partial N$ and $\{\vert \psi\vert =R\}$.
Then, it follows that 
\begin{align}
\int\limits_{\partial N}\vert \mathrm{D}_{h}v\vert_{h}\, d\sigma_{h}=\int\limits_{\partial N}h(\mathrm{D}_{h}v,\nu_{h})\,d\sigma_{h}=-\int\limits_{\{\vert \psi\vert =R\}}h(\mathrm{D}_{h}v, \nu_{h})\, d\sigma_{h}=-\int\limits_{\partial B_{R}}\widetilde{h}\Big(\mathrm{D}_{\widetilde{h}}\widetilde{v},\nu_{\,\widetilde{h}}\Big)\,d\sigma_{\widetilde{h}}\,,
\end{align}
where $\widetilde{v}=v\circ \psi^{-1}$. Now, 
thanks to~\eqref{eq5},~\eqref{eq37app},~\eqref{eq42app} and
~\eqref{eq41app}, and also by identity~\eqref{eq31} and
the second in \eqref{eq32}, and keeping in mind that $\vert\partial_{i}\widetilde{v} \vert\leq C\vert x\vert^{1-n}$,
we have that
\begin{align}
\int\limits_{\partial B_{R}}\widetilde{h}\Big(\mathrm{D}_{\widetilde{h}}\widetilde{v},\nu_{\,\widetilde{h}}\Big)\,d\sigma_{\widetilde{h}}&=\int\limits_{\partial B_{R}}g_{\R^{n}}\Big(\mathrm{D}_{e}\widetilde{v},\nu_{e}\Big)\,d\sigma_{e}+O(R^{-p})\\
&=\int\limits_{\partial B_{R}}g_{\R^{n}}\Bigg(\mathrm{D}_{e}\Big(\mathcal{C}\vert x\vert^{2-n}+z\Big),\frac{x}{\vert x\vert}\Bigg)\,d\sigma_{e}+O(R^{-p})\\
&=\int\limits_{\partial B_{R}}g_{\R^{n}}\Bigg(\mathrm{D}_{e}\Big(\mathcal{C}\vert x\vert^{2-n}\Big),\frac{x}{\vert x\vert}\Bigg)\,d\sigma_{e}+O(R^{-p})\\
&=-\mathcal{C}(n-2)\vert\SSS^{n-1}\vert+O(R^{-p}).
\end{align}
Hence 
$$\int\limits_{\partial N}\vert \mathrm{D}_{h}v\vert_{h}\, d\sigma_{h}=-\lim_{R\to \infty}\int\limits_{\partial B_{R}}\widetilde{h}\Big(\mathrm{D}_{\widetilde{h}}\widetilde{v},\nu_{\,\widetilde{h}}\Big)\,d\sigma_{\widetilde{h}}=\mathcal{C}(n-2)\vert\SSS^{n-1}\vert\,.$$
\end{proof}

\subsection{Measure of and integration on the level sets of the potential}\label{measandintlevelset}

\noindent
Let $(N,h)$ and $\iota:S \hookrightarrow N$ be respectively a $m$--dimensional Riemannian manifold and a $s$--dimensional Riemannian submanifold of $N$.
Let $k$ be a positive real number. We set
\begin{enumerate}[label=]
\item $\mathcal{B}(S)$ the smallest $\sigma$--algebra containing all open sets of $S$;
\item $(S,\Lambda(S),\mu_{\iota^{*}h})$ the canonical space of measure on the Riemannian manifold $(S,\iota^{*}h)$ (see~\cite[Section~3.4]{grigor});
\item $\mathcal{H}^{k}_{S}$ the $k$-dimensional Hausdorff measure on $(S,d_{S})$, being $d_{S}$ the distance function of $S$;
\item $\mathcal{H}^{k}_{S;N}$ the $k$-dimensional Hausdorff measure on $(S,d_{S;N})$ where $d_{S;N}$ is the distance function of ${N}$ restricted to $S\times S$;
\item \( \mathcal{H}^{k}_{N} \mres S\) the $k$-dimensional Hausdorff measure of $N$ restricted to $S$.
\end{enumerate}
By definition of the Hausdorff measure and by~\cite[Proposition~12.7]{Taylor}, $\mathcal{H}^{k}_{S;N}$, \(\mathcal{H}^{k}_{N} \mres S$ and $\mathcal{H}^{k}_{S}$  coincide on $\mathcal{B}(S)$, and by~\cite[Proposition~12.6]{Taylor} and by~\cite[Proposition~2.17]{Rudin}, $\mathcal{H}^{s}_{S}$ and $\mu_{\iota^{*}h}$ coincide on $\Lambda(S)$.
The same results still hold when $N$ is a manifold with boundary.\\
\smallskip

\noindent
For the ease of the reader, we collect in the next theorem 
some results about the measure of the level sets of the potential 
$u$ and 
\[
\mathrm{Crit}(u):=\{\vert \mathrm{D}_{g_{0}} u\vert_{g_{0}}=0\},
\] 
which are well--known in the Euclidean setting
(see, e.g.,~\cite{Hardt2} and~\cite{Hardt1}).

\begin{theorem}\label{geometryoflevelset}
Let $(M,g_{0},u)$ be a sub-static harmonic triple. Then the following statements hold true.
\begin{enumerate}[ label=(\roman*)]
  \item For every $t\in[0,1)$, the level set $\{u=t\}$ is compact and has finite $(n-1)$--Hausdorff measure in $M$;
  \item $\mathrm{Crit}(u)$ is a compact subset of $M$ and its Hausdorff dimension in $M$ is less than or equal to $(n-2)$;
  \item The set of the critical values of $u$ has zero Lebesgue measure, and for every $t\in[0,1)$ regular value of $u$ there exists $\epsilon_{t}>0$ such that $(t-\epsilon_{t},t+\epsilon_{t})\cap [0,1)$ does not contain any critical value of $u$.
\end{enumerate}
\end{theorem}
\begin{proof}
Each level set of $u$ is compact, due to the forth condition in~\eqref{f0}. Now, consider the nontrivial case where 
$\mathrm{Crit}(u)\neq\emptyset$ and let $p$ be a point of a
critical level set $\{u=t\}$. Take a chart $(U_{p},\psi_{p})$ centred at the point $p$ with $\psi_{p}(U_{p})=B_{1}$. 
Setting $\widetilde{g}_{0}=(\psi_{p})_{*}g_{0}=\widetilde{g}_{0;ij}dx^{i}\otimes dx^{j}$, 
there exists $C>0$ such that
\begin{equation}\label{f3}
C ^{-1} v^{i}v^{j}\delta _{ij}\leq\widetilde{g}_{0;ij}(x)v^{i}v^{j}\leq C v^{i}v^{j}\delta _{ij}, 
\end{equation}
for each $(v^{1},\dots,v^{n})\in\R^{n}$ and for each $x\in \overline{B_{\frac{1}{2}}}\,$. 
The same 
condition is
satisfied by the coefficients $\widetilde{g}_{0}^{\,ij}$.
In particular, setting $\widetilde{u}=u\circ\psi_{p}^{-1}$, 
we have that
$$\{ \vert \mathrm{D}_{\widetilde{g}_{0}} \widetilde{u}\vert_{\widetilde{g}_{0}} =0\}\cap B_{\frac{1}{2}}=\{ \vert \mathrm{D}_{e}\widetilde{u}\vert_{e}=0\}\cap B_{\frac{1}{2}}\,.$$
We observe that $\Delta_{\widetilde{g}_{0}}v=a_{ij}\partial_{i}\partial_{j}\widetilde v+b_{i}\partial_{i}\widetilde v=0$  is an elliptic partial differential equation with coefficient $C^{\infty}$ in $\overline{ B_{\frac{1}{2}}}$. 
We recall that if $v$ is a $C^{\infty}$--solution of the above equation 
and if $v$ vanishes to infinite order at a point $x_{0}\in B_{\frac{1}{2}} $, i.e. for every $k>0$
$$\lim_{r\to 0}\frac{1}{r^{k}}\int\limits_{B_{r}(x_{0})} v^{2}\,dx=0\,,$$ 
then $v$ is identically zero in $B_{\frac{1}{2}}$ (see~\cite[Theorem~ 1.2]{Garofalo}). 
Applying this fact to $\widetilde{u}-t$, one can argue that 
$\widetilde{u}-t$ has finite order of vanishing at $O$.
Then, by using ~\cite[Theorem~1.7]{Hardt1}, there exists $0<\rho <\frac{1}{2}$ such that
 $$ \mathcal{H}^{n-1}(B_{\rho}\cap \{\widetilde{u}=t\})< \infty\,.$$
Since $\widetilde{u}$ is nonconstant in $B_{\frac{1}{2}}$ and by the structure and regularity of $\Delta_{\widetilde{g}_{0}}$,~\cite[Theorem~1.1]{Hardt2} yields
$$ \mathcal{H}^{n-2}(B_{\rho}\cap \{\vert \mathrm{D}_{e}\widetilde{u}\vert_{e}=0\})< \infty\,.$$
Hence, since the restriction of $\psi_{p}$ to $\psi_{p}^{-1}(B_{1/2})$ is bilipschitz due to~\eqref{f3} and since  
the measures $\mathcal{H}^{k}_{\psi_{p}^{-1}(B_{1/2})}$ and 
$ \mathcal{H}^{k}_{M} \mres \psi_{p}^{-1}(B_{1/2})$ coincide on borel sets,
statements $(i)$ and $(ii)$ are true locally. 
In turn, by compactness of $\mathrm{Crit}(u)$, they are true globally.
To prove $(iii)$, observe that by Sard's Theorem, the set of the critical values of $u$ has zero Lebesgue measure. Now, suppose by contradiction that there exists $\overline{t}\in[0,1)$ regular value such that, for all $m\geq \overline{m}$ with $\frac{1}{\overline{m}}<1-\overline{t}$, the interval $(\overline{t}-\frac{1}{m},\overline{t}+\frac{1}{m})\cap [0,1)$ contains critical values. Hence there is a sequence $\{t_{m}\}_{m\geq \overline{m}}$ of critical values such that $t_{m}\to\overline{t}$. In particular, there exists a sequence 
$\{p_{m}\}_{m\geq \overline{m}}$ of critical points contained in the set $\{0\leq u\leq \overline{t}+\frac{1}{m}\}$ and such that $u(p_{m})=t_{m}$. 
Then, by compactness and up to a subsequence,
$p_{m}\to p$. In turn,
$0=\vert \mathrm{D}_{g_{0}} u\vert_{g_{0}} (p_{m})\to \vert\mathrm{D}_{g_{0}} u\vert_{g_{0}} (p)$ and $t_{m}=u(p_{m})\to u(p)$.
Hence $\vert\mathrm{D}_{g_{0}} u\vert_{g_{0}} (p)=0$ and $u(p)=\overline{t}$, which is absurd. This concludes the proof of $(iii)$.
\end{proof}

\begin{remark}\label{disposizionelevelset}
{\em It is useful to observe that:
\begin{itemize}
\item[(i)]
for every $t\in(0,1)$, the set $\{u\geq t\}$ is connected;
\item[(ii)] 
for every $t\approx1$, the level set $\{u=t\}$ is 
regular and diffeomorphic to $\SSS^{n-1}$;
\item[(iii)] 
For every $t\in(0,1)$, $\{u\geq t\}=\overline{\{u>t\}}$ and $\{0\leq u\leq t\}=\overline{\{0<u<t\}}$.
\end{itemize}}
\noindent
We check $(ii)$ first. We start by observing that due to \eqref{eq9} 
$\vert\mathrm{D}_{g_{0}} u\vert_{g_{0}}\neq0$
in $\{u\geq t_0\}$, for some $0<t_0<1$. 
This fact establishes a diffemorphism between
$\{u\geq t_0\}$ and $\{u=t_0\}\times[t_0,1)$
and tells us at the same time that the level sets $\{u=t\}$
are pairwise diffeomorphic, for every $t\geq t_0$.
It is thus sufficient to show that $\{u=t_0\}$ is connected.
Suppose by contradiction that this is not the case. Without loss of generality we can assume that
$\{u=t_0\}$ can be decomposed into the disjoint union of 
two connected sets $C_1$ and $C_2$, indeed the same argument works {\em a fortiori} if the connected components are more than two.
Now, note that by definition of asymptotically
flat manifold, there exists a compact set $K\subset M$
such that $M\setminus\mathring K$ is diffeomorphic to 
$\R^n\setminus\mathring B$ by a chart at infinity $\psi$,
where $B$ is a suitable ball, and we can suppose, 
up to a bigger $t_0$, that $\{u\geq t_0\}\subseteq M\setminus\mathring K$.
Now, in view of the asymptotic expansion of $u$, there exist two positive constants $A<B$ such that
\[
\frac A{|x|^{n-2}}\,\leq\, 1-u\,\leq\,\frac B{|x|^{n-2}}.
\]
In particular, setting $R_0=[B/(1-t_0)]^{1/(n-2)}$,
we have that
\[
\{|x| >R_0\}\,\subseteq\,\{u\geq t_0\}\,\simeq\,
\big\{C_1\times[t_0,1)\big\}\sqcup\big\{C_2\times[t_0,1)\big\}.
\]
At the same time, we have that $\{|x| > R_0\}$ is connected and each $C_i\times[t_0,1)$ 
is a closed set of $M$,
so that indeed
$\{|x| >R_0\}\,\subseteq\,
C_i\times[t_0,1),$
for some $i\in\{1,2\}$.Therefore, we have that
\begin{align*}
\big\{C_1\times[t_0,1)\big\}\sqcup\big\{C_2\times[t_0,1)\big\}
&\,=\,\{u\geq t_0\}
\,\subseteq\,M\setminus\mathring K\\
&\,=\,
\big[(M\setminus\mathring K)\cap\{|x|\leq R_0\}\big]
\sqcup
\big[(M\setminus\mathring K)\cap\{|x|>R_0\}\big]\\
&\,\subseteq\,
\big[(M\setminus\mathring K)\cap\{|x|\leq R_0\}\big]
\sqcup
\big\{C_i\times[t_0,1)\big\},
\end{align*}
which gives the contradiction
that the noncompact set 
$\{C_j\times[t_0,1)\big\}$, where $j\in\{1,2\}\setminus\{i\}$,
is contained into
the compact one
$(M\setminus\mathring K)\cap\{|x|\leq R_0\}$.
Therefore, $\{u=t_0\}$ is connected.
Now, $\{\widetilde{u}=t_{0}\}$, with $\widetilde{u}:=u\circ\psi^{-1}$, is a compact and connected hypersurface of $\R^{n}$, having strictly positive sectional curvature,
as Riemannian submanifold of $(\R^{n},g_{\R^{n}})$ , 
up to a bigger $t_{0}$ and due to~\eqref{eq13}. 
Hence, $\{\widetilde{u}=t_{0}\}$ is diffeomorphic to $\SSS^{n-1}$ by the 
Gauss map (see~\cite[Section~5.B]{GHF} for more details). 
Statement $(ii)$ thus follows, being $\{u=t_0\}$ and $\{\widetilde{u}=t_{0}\}$ diffeomorphic.

\smallskip
\noindent
To see $(i)$, observe first that if $t$ is a regular value of $u$, then $E_{t}:=\{u\geq t\}$ is a $n$--dimensional submanifold with boundary $\{u=t\}$. By Theorem~\ref{geometryoflevelset} and by the Maximum Principle, every connected component $C$ of $E_{t}$ is unbounded. Since $u\to 1$ at $\infty$, we have that $u(C)=[t,1)$, and hence 
$C\cap\{u=t_{0}\}\neq\emptyset$, for every $t_{0}\in(t,1)$. Then, $E_{t}$ is connected by $(ii)$.
If $t$ is a critical value of $u$, we let $\overline{t}>t$ be a regular value of $u$ 
such that $\{u=\overline{t}\}$ is connected 
and let $\{t_{m}\}$ be a nondecreasing sequence of regular value of $u$ such that 
$t_{m}<t$ and $t_{m}\to t$. 
Hence, $\big\{ \,\{t_{m}\leq u\leq \overline{t}\}\,\big\}_{m\in \N}$ is a nonincreasing family of connected and compact sets in $M$, which is Hausdorff, and in turn the intersection 
$\{t\leq u\leq \overline{t}\}$ is still connected.
In particular, we deduce that $E_t=\{t\leq u\leq \overline{t}\}\cup \{u\geq\overline{t}\}$ is connected.

\smallskip
\noindent
To check $(iii)$, note first that for every $t\in(0,1)$ regular value of $u$, the equalities are always true. If $t\in(0,1)$ is a critical value of $u$, by Theorem~\ref{geometryoflevelset} and by the Maximum Principle the interior of $\{0\leq u\leq t\}$ and the interior of $\{u\geq t\}$ are both disjoint from $\{u=t\}$, so that $(iii)$ is still true.
\end{remark}

\noindent
Let $(M,g_{0},u)$ be a sub-static harmonic triple, and let $t\in[0,1)$ be a real number.
We consider the spaces of measure 
$$\Big(\,\{u=t\},\mathcal{B}(\{u=t\}),\mathcal{H}^{n-1}_{M}\mres \{u=t\}\,\Big)\,\, \text{and}\,\, \Big(\,S:=\{u=t\}\setminus\mathrm{Crit}(u), \Lambda(S),\sigma_{g_{0}}:=\mu_{\iota^{*}g_{0}}\,\Big)\,.$$
Let $f:S\to \R$ be a continuous and consider the zero--extension $\mathring{f}$ of $f$ defined on $\{u=t\}$ as 
\begin{equation}
\mathring{f}(p)=
\begin{cases}
f(p)&\,\, \text{if $p\in S$},\\
0&\,\, \text{if $p\in \{u=t\}\cap\mathrm{Crit}(u)$}.\\
\end{cases}
\end{equation}
By Theorem~\ref{geometryoflevelset} and by definition of the Lebesgue integral, we have that $\mathring{f}\in L^{1}\big(\mathcal{H}^{n-1}_{M}\mres \{u=t\}\big)$ iff $f \in L^{1}(\sigma_{g_{0}})$ and
\begin{equation}
\label{caso1}
\int\limits_{\{u=t\}}\mathring{f}\,d\big(\mathcal{H}^{n-1}_{M}\mres \{u=t\}\big)=\int\limits_{S}f\,d\sigma_{g_{0}}.
\end{equation}
Similarly, if $f:\{u=t\}\to \R$ is a continuous function, then $f\in L^{1}\big(\mathcal{H}^{n-1}_{M}\mres \{u=t\}\big)$ iff $f|_{S}\in L^{1}(\sigma_{g_{0}})$ and 
\begin{equation}
\label{caso2}
\int\limits_{\{u=t\}} f \,d\big(\mathcal{H}^{n-1}_{M}\mres \{u=t\}\big)=\int\limits_{S}f|_{S}\,d\sigma_{g_{0}}.
\end{equation}
{\em In the rest of this paper, we will confuse the integrals
of cases \eqref{caso1} and \eqref{caso2}, denoting both 
by $\int\limits_{\{u=t\}}f  d\sigma_{g_{0}}$.}


\section{Monotonicity and Outer Rigidity Theorem}
\label{MonOutRigTheor}
In this section, we state and prove our Monotonicity and Outer Rigidity Theorem, 
which is then used to prove
the Capacitary Riemannian Penrose Inequality~\eqref{eq15}.
{\em From now on and unless otherwise stated, $(M,g_{0},u)$ will always be a sub-static harmonic triple, and, when referring to such triple, the subscript $g_{0}$ will be dropped. The only exception is $\vert \SSS^{n-1}\vert$, which always stands for the Euclidean volume of $\SSS^{n-1}$.}

\begin{theorem}[Monotonicity and Outer Rigidity Theorem]
\label{Monotonicity--Partial Rigidity Theorem}
Let $(M,g_{0},u)$ be a sub-static harmonic triple, and let $ F_{\beta}:\tau \in[1,+\infty)\to [0,+\infty)$ be the function defined by 
\begin{align}\label{F}
F_{\beta}(\tau):=(1+\tau)^{\beta\, \frac{n-1}{n-2}}\,\int\limits_{\big\{u=\sqrt{\frac{\tau-1}{\tau+1}}\,\big\}}\,\,\vert \mathrm{D}u \vert^{\beta+1}\,d\sigma\,,
\end{align}
for every $\beta\geq0$. Then, the following properties hold true.
\begin{enumerate}[ label=(\roman*)]
\item \underline{Differentiability, Monotonicity and Outer Rigidity}: for every $\beta>\frac{n-2}{n-1}$, the function $F_{\beta}$ is continuously differentiable with nonpositive derivative in $(1,\infty)$.
Moreover, if there exists $\tau_{0} \in (1,\infty)$ such that $F'_{\beta}(\tau_{0})=0$ for some $\beta >\frac{n-2}{n-1}$, then 
$(\,\{u\geq t_{0}\}, g_{0}\,)$, 
where $t_{0}=\sqrt{\frac{\tau_{0}-1}{\tau_{0}+1}}$, is isometric to
$$\Big(\,[r_{0},+\infty)\times \SSS^{n-1}\,, \frac{dr \otimes dr}{1-2\mathcal{C}r^{2-n}}+r^{2}g_{\SSS^{n-1}} \Big)\,,\,\,\quad r_{0}=[\mathcal{C}(1+\tau_{0})]^{\frac{1}{n-2}}\,.$$ 
\item \underline{Convexity}: for every $\beta>\frac{n-2}{n-1}$, the function $F_{\beta}$ is convex on $[1,\infty)$.
\end{enumerate}
\end{theorem}

\noindent
We remark that the functions {\em$F_{\beta}$ are well--defined}, in view of Theorem~\ref{geometryoflevelset} and since the integrand function in~\eqref{F} is bounded on every level set of $u$.
Note that, once Theorem~\ref{Monotonicity--Partial Rigidity Theorem} is proven, The Dominated Convergence Theorem allows us to extend the monotonicity and convexity of $F_{\beta}$ also to the case $\beta=\frac{n-2}{n-1}$\,.
Moreover, on the values $\tau$ such that $\Big\{u=\sqrt{\frac{\tau-1}{\tau+1}}\Big\}$ is regular and thus on a.e. $\tau>1$ due to Theorem~\ref{geometryoflevelset} $(iii)$, the function $F_{\beta}$ is twice differentiable for each $\beta>\frac{n-2}{n-1}$, with first and second derivative given by 
\begin{align}
F'_{\beta}(\tau)&=-\beta\,\frac{(\tau+1)^{\beta\frac{n-1}{n-2}-\frac{3}{2}}}{\sqrt{\tau-1}}\int\limits_{\big\{u=\sqrt{\frac{\tau-1}{\tau+1}}\,\big\}}\,\vert \mathrm{D}u \vert^{\beta}\Big[\mathrm{H}-\frac{n-1}{n-2}\,\, \frac{2u}{1-u^{2}}\,\,\vert \mathrm{D}u \vert\Big]d\sigma\,,
\label{derivata_formale}\\
F''_{\beta}(\tau)&=
\beta \frac{(\tau+1)^{\beta\frac{n-1}{n-2}-3}}{\tau-1} \Bigg\{\Big(\beta-\frac{n-2}{n-1}\Big)\int\limits_{\big\{u=\sqrt{\frac{\tau-1}{\tau+1}}\,\big\}}\vert \mathrm{D}u \vert^{\beta-1} \Big[\,\mathrm{H}-\frac{n-1}{n-2}\,\,\frac{2u}{1-u^{2}}\,\,\vert \mathrm{D}u \vert\,\Big]^{2}\,d\sigma\\
&+\beta \int\limits_{\big\{u=\sqrt{\frac{\tau-1}{\tau+1}}\,\big\}}\vert \mathrm{D}u \vert^{\beta-3}\,\,\big\vert \,\mathrm{D}^{\mathrm{T}}\,\vert \mathrm{D}u \vert\,\big\vert^{2}d\sigma
+\int\limits_{\big\{u=\sqrt{\frac{\tau-1}{\tau+1}}\,\big\}}\vert \mathrm{D}u \vert^{\beta-1}\,\,\Big[\,\vert \mathrm{h}\vert^{2}\,-\,\frac{1}{n-1}\,\,\vert \mathrm{H}\vert^{2}\, \Big]\, d\sigma\\
&+\int\limits_{u=\sqrt{\frac{\tau-1}{\tau+1}}\,\big\}}\vert \mathrm{D}u \vert^{\beta-1}\,\Bigg[\,\mathrm{Ric}(\nu,\nu)-\frac{\mathrm{D}^{2}u\,(\nu,\nu)}{u}\,\Bigg]\,d\sigma\,\,\Bigg\}\, .\label{F''}  
\end{align}
In the computation, we have used the first normal variation of the 
volume and the mean curvature of $\{u=t\}$, and the Divergence Theorem.
The symbols $\mathrm{H}$ and $\mathrm{h}$ stand respectively for
the mean curvature and the second fundamental form of the smooth $(n-1)$--dimensional submanifold $\big\{u=\sqrt{\frac{\tau-1}{\tau+1}}\,\big\}$, with respect to the $\infty$--pointing unit normal vector field $\nu=\frac{\mathrm{D}u}{\vert \mathrm{D}u \vert}$. Also, $\mathrm{D}^{\mathrm{T}}$ denotes the tangential part of the gradient, 
that is
$$\mathrm{D}^{\mathrm{T}}f =\mathrm{D}f-g_{0}(\mathrm{D}f,\nu)\nu,\,$$ 
for every $f\in C^{1}(M)$.\\
\medskip

\noindent
To prove Theorem~\ref{Monotonicity--Partial Rigidity Theorem},
we use the results of Section~\ref{conformalsetting}, which are obtained in the conformal
setting defined by
\begin{align}\label{eq50}
g=(1-u^{2})^{\frac{2}{n-2}}g_{0}\,,\,\,\quad\quad\quad\,\, \varphi=\log\Big(\frac{1+u}{1-u}\Big)\,.
\end{align}
Denoting by $\nabla$ and $\Delta_{g}$ the Levi--Civita connection and the Laplace--Beltrami operator of $g$, the
triple $(M,g,\varphi)$ satisfies the following system.
\begin{equation}\label{sistconf}
\left\{
\begin{aligned}
\mathrm{Ric}_{g}-\coth(\varphi)\,\nabla ^{2}\varphi+\frac{1}{n-2}\,d\varphi \otimes d\varphi-\frac{1}{n-2}\,\vert \nabla \varphi \vert^{2}_{g}\,\ g \,&\geq0 &&\mathrm{in}\, \mathring{M},\\
\Delta_{g}\varphi  &=0   &&\mathrm{in} \,  M,\\
\varphi &=0 &&\mathrm{on} \,  \, \partial M,\\
\varphi &\to +\infty\,&&\mathrm{at}\,\infty.
\end{aligned}
\right .
\end{equation}

\noindent
Moreover, we have that
\begin{align}
\vert \nabla \varphi \vert_{g}^{2}= 4\,\vert \mathrm{D}u \vert^{2}(1-u^{2})^{-2\frac{n-1}{n-2}}\longrightarrow\big( 2\,\mathcal{C}\big)^{-\frac{2}{n-2}}(n-2)^{2}\,\,\text{at}\, \,\infty\,,\,\tag{$\star$}
\end{align}
as we will see in the proof of Lemma~\ref{lemmabound}.

\begin{remark}\label{remark2}
Since $\mathrm{Crit}(\varphi)=\mathrm{Crit}(u)=\{\vert \nabla \varphi \vert_{g}=0\}$ by the equality in~($\star$) and since $\{\varphi=s\}=\{u=\tanh\big(\frac{s}{2}\big)\}$ by~\eqref{eq50}, using Theorem~\ref{geometryoflevelset} and Remark~\ref{disposizionelevelset}
we deduce that: {\em $\mathrm{Crit}(\varphi)$ has zero $\mu_{g}$--measure and zero $(n-1)$--Hausdorff measure in $(M,g)$; the level sets of $\varphi$ have finite $(n-1)$--Hausdorff measure in $(M,g)$ and in particular the smooth $(n-1)$--dimensional submanifolds $\{\varphi=s\}\setminus\mathrm{Crit}(\varphi)$ have finite $g$--area, i.e. finite $\sigma_{g}$--measure.
Moreover, $\{\varphi\geq s\}$ is connected for every $s\geq 0$ and there exists $s_0\geq0$
such that $\{\varphi =s\}$ is regular and diffeomorphic to $\SSS^{n-1}$,
for every $s\geq s_0$.
Similar comments as those at the end of Subsection~\ref{measandintlevelset} hold,
regarding the relation between integration and $\mathrm{Crit}(\varphi)$.
}
\end{remark}
\medskip

\noindent
Let $\Phi_{\beta}:[0,\infty)\to \R$ be the function defined by formula
\begin{align}\label{Phibetasectu}
\Phi_{\beta}(s):=\int\limits_{\{\varphi=s\}}\vert \nabla \varphi \vert^{\beta+1}_{g}\,d\sigma_{g}\,,
\end{align}
for every $\beta\geq 0$. 
For the convenience of the reader, we anticipate from Section~\ref{conformalsetting}
the properties of $\Phi_{\beta}$ that we are going to use.
\begin{enumerate}
\item[($\circ$)] For every $\beta\geq0$, the function $\Phi_{\beta}(s)$ is continuous in $[0,+\infty)$.
\item[($\diamond$)] For every $\beta>\frac{n-2}{n-1}$, the function $\Phi_{\beta}$ is continuously differentiable in $(0,+\infty)$. The derivative $\Phi'_{\beta}$  is nonpositive, satisfies for every $S>s>0$ 
\begin{align}\label{eq52}
\frac{\Phi_{\beta}'(S)}{\sinh(S)}-\frac{\Phi_{\beta}'(s)}{\sinh(s)}\geq0\,,
\end{align}
 and admits for every $s>0$ the integral representation
\begin{align*}
\Phi'_{\beta}(s)=-\beta\, \sinh( s) \,\int\limits_{\{\varphi>s\}}\frac{\vert \nabla \varphi \vert^{\beta-2}_{g}\,\Big[\,(\beta-2)\,\Big\vert\nabla\vert \nabla \varphi \vert_{g}\,\Big\vert_{g}^{2}+\vert \nabla^{2} \varphi \vert_{g}^{2}+Q(\nabla \varphi ,\nabla \varphi )\, \Big]}{\sinh\varphi}\,d\mu_{g}\leq 0\,.
\end{align*}
\item[($\diamond\,\diamond$)] If there exists $s_{0}>0$ such that $\Phi'_{\beta}(s_{0})=0$ for some $\beta>\frac{n-2}{n-1}$, then $\{\varphi=s_{0}\}$ is connected and $(\{\varphi\geq s_{0}\},g)$ is isometric to $\big([0,+\infty)\times 
\{\varphi=s_{0}\},d\rho\otimes d\rho+g_{\{\varphi=s_{0}\}})$, where $\rho$ is the $g$--distance function to $\{\varphi=s_{0}\}$ and $\varphi$ is an affine function of $\rho$ in $\{\varphi\geq s_{0}\}$.
If $\Phi_{\beta}$ is constant for some $\beta>\frac{n-2}{n-1}$, then $\partial M$ is connected and 
$(M,g)$ is isometric to $\big([0,+\infty)\times \partial M,d\rho\otimes d\rho+g_{\partial M})$, where $\rho$ is the $g$--distance function to $\partial M$ and $\varphi$ is an affine function of $\rho$.
\end{enumerate}
In the above list we have gathered and summarised the results 
contained in Lemma~\ref{firstitegralidenty}, Proposition~\ref{Conforme Monotonicity theorem}, and Corollary~\ref{corConforme Monotonicity theorem}.
\medskip

\noindent
\begin{proof}[\underline{Proof of Theorem~\ref{Monotonicity--Partial Rigidity Theorem}. Step $1$: Differentiability, Monotonicity and Convexity}.]
For every $\beta\geq0$ and for all $\tau \in[1,+\infty)$, we note that
\begin{align}\label{eq59}
F_{\beta}(\tau)=2^{\frac{\beta}{n-2}-1}\,\, \Phi_{\beta}\Bigg(\log\Bigg(\,\frac{\sqrt{\tau+1}+\sqrt{\tau-1}}{\sqrt{\tau+1}-\sqrt{\tau-1}}\,\Bigg) \Bigg) \,.
\end{align}
Consequently, by $(\circ)$, we deduce that {\em for every $\beta\geq0$ the function $F_{\beta}$ is continuous in $[1,+\infty)$}.
By $(\diamond)$ and by~\eqref{eq59}, we obtain immediately that for every $\beta > \frac{n-2}{n-1}$ the function $F_{\beta}$ is continuously differentiable in $(1,+\infty)$, with
\begin{align} \label{F'betainterminiPhi'beta}
F'_{\beta}(\tau)=2^{\frac{\beta}{n-2}-1}\,\frac{1}{\sqrt{\tau^{2}-1}}\,\Phi'_{\beta}\Bigg(\log\Bigg(\,\frac{\sqrt{\tau+1}+\sqrt{\tau-1}}{\sqrt{\tau+1}-\sqrt{\tau-1}}\,\Bigg) \Bigg)\,.
\end{align}
In particular, from $(\diamond)$ we get $F'_{\beta}\leq 0$.
As for the convexity, observing
$$\frac{\Phi'_{\beta}\Big(\log\Big(\,\frac{\sqrt{\tau+1}+\sqrt{\tau-1}}{\sqrt{\tau+1}-\sqrt{\tau-1}}\,\Big) \Big)}{\sinh \big(\log\big(\,\frac{\sqrt{\tau+1}+\sqrt{\tau-1}}{\sqrt{\tau+1}-\sqrt{\tau-1}}\,\big) \big)}=\frac{\Phi'_{\beta}\Big(\log\Big(\,\frac{\sqrt{\tau+1}+\sqrt{\tau-1}}{\sqrt{\tau+1}-\sqrt{\tau-1}}\,\Big) \Big)}{\sqrt{\tau^{2}-1}}\,,$$
and that the function $\log\big(\,\frac{\sqrt{\tau+1}+\sqrt{\tau-1}}{\sqrt{\tau+1}-\sqrt{\tau-1}}\,\big)$ is nondecreasing, from~\eqref{eq52} we obtain that $F'_{\beta}\,$ is nondecreasing in $(1,+\infty)$, and then by the continuity, $F_{\beta}$ is convex in $[1,+\infty)$.
\bigskip

\noindent
{\em \underline{Step $2$: Outer Rigidity.}} Let us assume that there exists $\tau_{0} \in (1,\infty)$ such that $F'_{\beta}(\tau_{0})=0$ for some $\beta >\frac{n-2}{n-1}$.
Then, by equality~\eqref{F'betainterminiPhi'beta}, $\Phi'_{\beta}(s_{0})=0$ for $s_{0}=\log\big(\,\frac{\sqrt{\tau_{0}+1}+\sqrt{\tau_{0}-1}}{\sqrt{\tau_{0}+1}-\sqrt{\tau_{0}-1}}\,\big)$, and hence, 
by $(\diamond\,\diamond)$, $(\{\varphi\geq s_{0}\},g)$ is isometric to $\big([0,+\infty)\times \{\varphi=s_{0}\},d\rho\otimes d\rho+g_{\{\varphi=s_{0}\}})$, where $\rho$ is the $g$--distance function to $\{\varphi=s_{0}\}$ and 
$\varphi=(n-2)\, (2\mathcal{C})^{-\frac{1}{n-2}}\rho+s_{0}$, because $\vert \nabla \rho\vert_{g}\equiv1$ and in view of the limit in~$(\star)$.
Setting $t_{0}=\tanh \frac{s_{0}}{2}$, consider $N$ the submanifold with boundary $\{\varphi\geq s_{0}\}=\{u\geq t_{0}\}$. Writing
\begin{align}
&N &\,&  &&[0,+\infty)\times \partial N  &\,&  &&[s_{0},+\infty)\times \partial N     &\,&      &&[t_{0},1)\times \partial N \\ 
&g &\,&   &&d\rho\otimes d\rho+g_{\partial N}    &\,&       & &\frac{d\varphi\otimes d\varphi}{(n-2)^{2}\, (2\mathcal{C})^{-\frac{2}{n-2}}}+g_{\partial N}   &\,& &&\frac{2^{2\,\frac{n-1}{n-2}}\,\, \mathcal{C}^{\frac{2}{n-2}} }{(n-2)^{2}\,(1-u^{2})^{2} }\,\, du\otimes du  + g_{\partial N}\\
&p &\mapsto&   &&\big(\rho,q\big)        &\mapsto&     &&\big (\varphi=(n-2)\, (2\mathcal{C})^{-\frac{1}{n-2}} \,\rho, q\big)    &\mapsto&                      && \Big(u=\tanh\, \frac{\varphi}{2} , q\Big)\,,                    
\end{align}
the Riemannian manifolds in the first row, whose metrics are indicated in the second row, are pairwise isometric through the applications written in the third row. 
We recall that the application $p\to (\rho,q)$ in the third row is the inverse of the diffeomorphism given by the normal exponential map, i.e. the application which associates to every point $p$ of $N$ the couple having as first coordinate the $g$--distance of $p$ from $\partial N$ and as second coordinate the point $q$ of $\partial N$ that realizes such distance.
Then, in view of~\eqref{eq50} and with the same notation as above, the following Riemannian manifolds are isometric.
\begin{align}
&N  &\,&   &&[t_{0},1)\times \partial N  &\,&  &&[r_{0},+\infty)\times \partial N\\
&g_{0}  &\,&   &&\frac{2^{2\,\frac{n-1}{n-2}}\,\, \mathcal{C}^{\frac{2}{n-2}} }{(n-2)^{2}\,(1-u^{2})^{2\,\frac{n-1}{n-2}} }\,\, du\otimes du  +(1-u^{2})^{-\frac{2}{n-2}} g_{\partial N} &\,&  &&\frac{dr \otimes dr}{1-2\,\mathcal{C} \,r^{2-n}}+ (2\,\mathcal{C})^{-\frac{2}{n-2}}\, r^{2}g_{\partial N} \\
&p  &\mapsto&   && \big(u, q\big)     &\mapsto&                 && \Big (r=\Big(\,\frac{2\mathcal{C}}{1-u^{2}}\,\Big)^{\frac{1}{n-2}},q\Big)\,,\, \label{eq21}
\end{align}
where $r_{0}=\big(\frac{2\mathcal{C}}{1-t_{0}^{2}}\big)^{\frac{1}{n-2}}$.
Doing some computations, we obtain that
\begin{align}\label{eq22}
\vert \mathrm{Rm}\vert^{2}\,(p)=(2\mathcal{C})^{\frac{4}{n-2}} \,r^{-4}\,(p)\,\,\Big\vert \mathrm{Rm}_{g_{\partial N}}+\frac{1-2\,\mathcal{C} \,r^{2-n}}{2^{\frac{n}{n-2}}\,\,\mathcal{C}^{\frac{2}{n-2}}}\, g_{\partial N}\owedge g_{\partial N}\,\Big\vert _{g_{\partial N}}^{2}\,(q)+c\,r^{-2n}\,(p)\,,
\end{align}
where the convection followed for the Riemannian curvature tensor is that given in~\cite{petersen1}, $c$ is a suitable positive constant and $q$ is the point of $\partial N$ that realizes the $g$--distance of $p$ from $\partial N$.
Denoting by $ \Theta$ the diffeomorphism from $N$ to $[r_{0},+\infty)\times \partial N$ introduced in~\eqref{eq21}, for every $q_{0}\in \partial N$ we consider the curve 
$$\gamma:r\in [r_{0},+\infty)\to \Theta^{-1}(r,q_{0})\in M $$
and observe from~\eqref{eq22} that 
\begin{align}\label{eq23}
(2\mathcal{C})^{-\frac{4}{n-2}} r^{4}\,\vert \mathrm{Rm}\vert^{2}\,(\gamma(r)) \xrightarrow{r\rightarrow+\infty}\,\Big\vert \mathrm{Rm}_{g_{\partial N}}+\frac{(2\mathcal{C}\,)^{-\frac{2}{n-2}}}{2} g_{\partial N}\owedge g_{\partial N}\,\Big\vert _{g_{\partial N}}^{2}\,(q_{0})\,.
\end{align}
At the same time, we have that
\begin{align}\label{eq24}
r^{4}\,\vert \mathrm{Rm}\vert^{2}\,(\gamma(r)) \xrightarrow{r\rightarrow+\infty} 0\,.
\end{align}
This is because $g_{0}$ is asymptotically flat according to Definition~\ref{asymptoticallyflatmanifold} and by~\eqref{eq11}, which yields in particular
\begin{align}
\vert \mathrm{Rm}\vert =O(\vert x\vert_{\R^{n}}^{-(p+2)})\quad\quad\text{and}\quad\quad \frac{r}{\vert x\vert_{\R^{n}}}&\xrightarrow{\vert x\vert_{\R^{n}}\rightarrow+\infty}1\,,
\end{align}
for some $p>\frac{n-2}{2}$.
Combining~\eqref{eq23} and~\eqref{eq24}, the arbitrariness of the point $q_{0}$ in $\partial N$ gives that $$\mathrm{Rm}_{g_{\partial N}}=-\frac{(2\mathcal{C}\,)^{-\frac{2}{n-2}}}{2} g_{\partial N}\owedge g_{\partial N}\,.$$
Hence the sectional curvature of the Riemannian manifold $(\partial N, g_{\partial N})$ is constant and identically equal to 
$(2\mathcal{C}\,)^{-\frac{2}{n-2}}$. Then, being all the level sets 
$\{u=t\}$ with 
$t\approx1$ regular and diffeomorphic to $\SSS^{n-1}$ 
as observed in Remark~\ref{disposizionelevelset}, 
for~\cite[Section~3.F]{GHF} $(\partial N, g_{\partial N})$ and 
$(\SSS^{n-1}, (2\mathcal{C}\,)^{\frac{2}{n-2}}g_{\SSS^{n-1}})$ are isometric.
Then, $(\{u\geq t_{0}\},g_{0})$ is isometric to the submanifold $\big(\,[r_{0},+\infty)\times \SSS^{n-1}\,, \frac{dr \otimes dr}{1-2\mathcal{C}r^{2-n}}+r^{2}g_{\SSS^{n-1}} \big)$ of the Schwarzschild manifold with associated ADM mass given by $ \mathcal{C}$.
\end{proof}
\bigskip

\noindent
\begin{proof}[\underline{Proof of Theorem~\ref{capacitaryriemannianpenroseinequality}. Spep $1$: Inequality}.]
By Theorem~\ref{Monotonicity--Partial Rigidity Theorem}, we have that
$F_{\beta}(\tau_{0})\geq \lim_{\tau\to+\infty}\,F_{\beta}(\tau)$,
for every $\tau_{0}>1$. In particular, since $F_{\beta}$ is continuous in $[1,+\infty)$ due to the step $1$ of Theorem~\ref{Monotonicity--Partial Rigidity Theorem}, we have that
\begin{align}\label{eq18}
F_{\beta}(1)\geq \lim_{\tau\to+\infty}\,F_{\beta}(\tau)\,,
\end{align}
for every $\beta>\frac{n-2}{n-1}$. Since $\mathrm{D}^{2}u \equiv 0$ on $\partial M$ and since $\partial{M}$ of $M$ is connected, $\vert \mathrm{D} u \vert$ is constantly equal to 
$\frac{(n-2)\,\mathcal{C}\,\vert \SSS^{n-1}\vert_{\R^{n}}\,}{\vert \partial M\vert}\,$, by formula~\eqref{eq10}. In particular, we have that
\begin{align}\label{eq19}
F_{\beta}(1)=\frac{2^{\beta\,\frac{n-1}{n-2}} \,\,(n-2)^{\beta+1}\,\,\mathcal{C}^{\beta+1}\,\,\vert \SSS^{n-1}\vert_{\R^{n}}^{\beta+1}}{\vert \partial M\vert^{\beta}}
\end{align}
By $(\star)$, we know that
\begin{align}\label{eq16}
\frac{\vert \mathrm{D}u \vert\,\,\,\,\,\,}{\,\,(1-u^{2})^{\frac{n-1}{n-2}}}\, \longrightarrow \, 2^{-\frac{n-1}{n-2}}\, (n-2)\,\mathcal{C}^{-\frac{1}{n-2}}\quad \,\, \text{at}\,\,\infty 
\end{align}
Therefore, fixed $\varepsilon>0$, there exists $1< \tau_{0}<+\infty$ such that 
\begin{align}\label{eq17}
\vert \mathrm{D}u \vert\geq \,(1-u^{2})^{\frac{n-1}{n-2}}\Big( 2^{-\frac{n-1}{n-2}}\, (n-2)\,\mathcal{C}^{-\frac{1}{n-2}}-\varepsilon\Big)
\end{align}
in $\big\{u\geq\sqrt{\frac{\tau_{0}-1}{\tau_{0}+1}}\,\big\}$ and the level sets $\big\{u=\sqrt{\frac{\tau-1}{\tau+1}}\,\big\}$ are regular for all $\tau\geq \tau_{0}$.
Therefore, for every $\tau\geq \tau_{0}$ we have that
\begin{align}
F_{\beta}(\tau)&=(1+\tau)^{\beta\, \frac{n-1}{n-2}}\,\int\limits_{\big\{u=\sqrt{\frac{\tau-1}{\tau+1}}\,\big\}}\,\,\vert \mathrm{D}u \vert^{\beta+1}\,d\sigma\\
&\geq (1+\tau)^{\beta\, \frac{n-1}{n-2}}\,\int\limits_{\big\{u=\sqrt{\frac{\tau-1}{\tau+1}}\,\big\}}\,\, (1-u^{2})^{\beta\,\frac{n-1}{n-2}}\,\Big( 2^{-\frac{n-1}{n-2}}\, (n-2)\,\mathcal{C}^{-\frac{1}{n-2}}-\varepsilon\Big)^{\beta} \,\,\vert \mathrm{D}u \vert\,d\sigma\\
&=2^{\beta\,\frac{n-1}{n-2}}\,\Big( 2^{-\frac{n-1}{n-2}}\, (n-2)\,\mathcal{C}^{-\frac{1}{n-2}}-\varepsilon\Big)^{\beta} \,\,\int\limits_{\partial M }\,\vert \mathrm{D}u \vert\,d\sigma\\
&=2^{\beta\,\frac{n-1}{n-2}}\,(n-2)\,\mathcal{C} \,\Big( 2^{-\frac{n-1}{n-2}}\, (n-2)\,\mathcal{C}^{-\frac{1}{n-2}}-\varepsilon\Big)^{\beta}\, \vert \SSS^{n-1}\vert_{\R^{n}}\,,
\end{align}
where in the second equality we have used the Divergence Theorem couple with the fact that $u$ is harmonic, and in the third equality we have used formula~\eqref{eq10}. Since $\varepsilon$ is arbitrary, we get
$$\lim_{\tau\to+\infty}\,F_{\beta}(\tau)\geq\,(n-2)^{\beta+1}\,\mathcal{C}^{1-\frac{\beta}{n-2}} \, \vert \SSS^{n-1}\vert_{\R^{n}}\,.$$
In a similar way we can obtain the reverse inequality, so that
\begin{align}\label{eq20}
\lim_{\tau\to+\infty}\,F_{\beta}(\tau)=\,(n-2)^{\beta+1}\,\mathcal{C}^{1-\frac{\beta}{n-2}} \, \vert \SSS^{n-1}\vert_{\R^{n}}\,.
\end{align}
Joining the formulas in~\eqref{eq18},~\eqref{eq19} and~\eqref{eq20}, we obtain the desired inequality~\eqref{eq15}.\\
\smallskip

\noindent
\underline{{\em Step $2$: Rigidity}}.
If $(M,g_{0})$ is isometric to the Schwarzschild manifold with ADM mass $m>0$, then the right--hand side and the left--hand side of~\eqref{eq15} are both equal to 
$m$, by direct computation.\\
Suppose now that the equality holds in~\eqref{eq15}. Then, 
by Step 1 and for every $\beta>\frac{n-2}{n-1}$, 
the function $F_{\beta}$ is constant.
In turn, $\Phi_{\beta}$ is constant, being
$$\Phi_{\beta}(s)=2^{1-\frac{\beta}{n-2}}\, F_{\beta}\Bigg(\,\frac{1+\tanh^{2}\big( \frac{s}{2}\big)}{1-\tanh^{2}\big( \frac{s}{2}\big)}\,\Bigg)\,.$$
Finally, $(\diamond\,\diamond)$ and the very same argument of the proof of 
the Outer Rigidity in 
Theorem~\ref{Monotonicity--Partial Rigidity Theorem} 
imply first that $(M,g)$ is isometric to 
$$\big([0,+\infty)\times \partial M,d\rho\otimes d\rho+g_{\partial M}),$$
where $\rho$ is the $g$--distance to $\partial M$ and $\varphi$ is an affine function of $\rho$, and secondly that $(M,g_{0})$ is isometric to the Schwarzschild manifold with ADM mass $\mathcal{C}$.
\end{proof}

\section{Conformal setting}\label{conformalsetting}
Let us consider the conformal change 
$g$ of the metric $g_0$
introduced in~\eqref{eq50}
which is well--defined being $0\leq u<1$ in $M$. 
The metric $g$ is complete, since any $g$--geodesic $\gamma$ parametrized by $g$--arc length defined on a bounded interval $[0,a)$ can be extended to a continuous path on $[0,a]$. 
Indeed, if $\gamma$ has infinity length with respect to $g_{0}$, there exists a sequence $\{t_{m}\}_{m\in \N}$ such that $\gamma(t_{m})\to \infty$ (being $\gamma$ not contained in any compact set) and using, in the computation of $g$--length of $\gamma$, the passage from $g$ to $g_{0}$, the asymptotic flatness of $(M,g_{0})$ and the asymptotic expansion of $u$ in~\eqref{eq11} we obtain that $\gamma$ has infinity length with respect to $g$.
Hence $\gamma$ has finite length with respect to $g_{0}$ and, being $g_{0}$ complete, it follows that $g$ is complete (see~\cite[Section~1.1]{Pigola} and~\cite{Dirmeier}).
We also recall
that the metric $g$ is {\em asymptotically cylindrical} (see~\cite[Section~ 3.1]{Virginia1}).
The other main element of the conformal setting is the $C^{\infty}$--function $\varphi$, defined in~\eqref{eq50}. Now, the reverse changes are
$$g_{0}=\Big(\cosh\frac{\varphi}{2}\Big)^{\frac{4}{n-2}}g\,,
\quad\quad\quad u=\tanh \frac{\varphi}{2}\,.$$
Recalling that we denote by the symbols $\nabla$ and $\Delta_{g}$ the Levi--Civita connection and the Laplace--Beltrami operator of $g$, by the formulas in~\cite[Theorem~1.159]{Besse}, we obtain
\begin{align}
\mathrm{D}u&=\frac{1}{2}\Big(\cosh \frac{\varphi}{2}\Big)^{-\frac{2n}{n-2}}\,\,\nabla \varphi\,,\label{eq51}\\
\mathrm{D}^{2}u&=\frac{1}{2}\, \frac{1}{\cosh^{2}\frac{\varphi}{2}}\,\nabla^{2}\varphi- \frac{n}{2(n-2)}\,\frac{\sinh\frac{\varphi}{2}}{\cosh^{3}\frac{\varphi}{2}}\, d\varphi \otimes d\varphi+\frac{1}{2(n-2)}\,\frac{\sinh\frac{\varphi}{2}}{\cosh^{3}\frac{\varphi}{2}}\,\vert \nabla \varphi \vert^{2}_{g}\,g\,,\label{eq53}\\
\Delta u&=\frac{1}{2}\Big(\cosh \frac{\varphi}{2}\Big)^{-\frac{2n}{n-2}}\Delta_{g}\varphi\,,\label{eq54}\\
\mathrm{Ric}&=\mathrm{Ric}_{g}-\tanh\Big(\frac{\varphi}{2}\Big)\,\nabla ^{2}\varphi+\Big[\frac{1}{n-2}\tanh^{2}\Big( \frac{\varphi}{2}\Big)-\frac{1}{2}\frac{1}{\cosh^{2}\frac{\varphi}{2}}\Big]d\varphi \otimes d\varphi \nonumber\\
&-\frac{1}{(n-2)}\,\Big[\,\frac{1}{2}\,\frac{1}{\cosh^{2}\frac{\varphi}{2}}+\tanh^{2}\Big( \frac{\varphi}{2}\Big)\Big]\vert \nabla \varphi \vert^{2}_{g}\,g\,.
\end{align}
Translating system~\eqref{f0} in terms of $g$ and $\varphi$, we get system~\eqref{sistconf}.
Moreover, on $\{\varphi=s\}\setminus\mathrm{Crit}(\varphi)$ we consider the $\infty$--pointing normal unit vector fields 
\begin{align}
\nu=\frac{ \mathrm{D}u }{\vert \mathrm{D}u \vert}\,, \quad \quad \quad \nu_{g}=\frac{\nabla \varphi}{\vert \nabla \varphi \vert_{g}}\,,
\end{align}
the mean curvatures 
\begin{align}
\mathrm{H}=-\, \frac{\mathrm{D}^{2}u(\mathrm{D}u,\mathrm{D}u)}{\vert \mathrm{D}u \vert^{3}}\,,\quad \quad \quad \mathrm{H}_{g}=-\,\frac{\nabla ^{2}\varphi(\nabla \varphi,\nabla \varphi)}{\vert \nabla \varphi \vert_{g}^{3}}\,,\,\label{H}
\end{align}
and the second fundamental forms 
\begin{align}
\mathrm{h}(X,Y)=\frac{\mathrm{D}^{2}u\,(X,Y)}{\vert \mathrm{D}u \vert}\,,\quad\quad\quad \mathrm{h}_{g}(X,Y)=\frac{\nabla ^{2}\varphi\,(X,Y)}{\vert \nabla \varphi \vert_{g}}\,,
\end{align}
for any $X,Y$ tangent vector fields to the considered submanifold.

\noindent Reversing formulas~\eqref{eq51},~\eqref{eq53} and~\eqref{eq54}, we get
\begin{align}
\nabla \varphi&=\frac{2}{(1-u^{2})^{\frac{n}{n-2}}}\, \mathrm{D}u\,,\\
\nabla ^{2}\varphi&=\frac{2}{1-u^{2}}\,\Big[\,\mathrm{D}^{2}u\,+\, \frac{n}{n-2}\,\frac{2u}{1-u^{2}}\, du\otimes du\,-\,\frac{1}{n-2}\,\frac{2u}{1-u^{2}}\,\vert \mathrm{D}u \vert^{2}\,g_{0}\,\Big]\,,\\
\vert \nabla^{2} \varphi \vert_{g}^{2}&=\frac{4}{(1-u^{2})^{\frac{2n}{n-2}}}\,\vert \mathrm{D}^{2}u \vert^{2}+\frac{16n}{n-2}\,\frac{u}{(1-u^{2})^{\frac{3n-2}{n-2}}}\,\mathrm{D}^{2}u(\mathrm{D}u,\mathrm{D}u)+\frac{16n(n-1)}{(n-2)^{2}}\,
\frac{u^{2}}{(1-u^{2})^{4\big(\frac{n-1}{n-2}\big)}}\,\vert \mathrm{D}u \vert^{4}\,.
\end{align}
These equalities, jointly with the asymptotic flatness of $(M,g_{0})$ and the asymptotic expansion of $u$ given in Section~\ref{prel}, allow us to obtain an upper bound for the functions $\vert \nabla \varphi \vert_{g}$ and $\vert \nabla^{2} \varphi \vert_{g}$, and for the $g$--areas of the level sets of $\varphi$ sufficiently "close" to infinity. This is the content of the following lemma.

\begin{lemma}\label{lemmabound}
There exists $0\leq s_{0}<+\infty$ such that
\begin{align}\label{boundutili}
\sup_{M}\vert \nabla \varphi \vert_{g}\,+\,\sup_{M}\vert \nabla^{2} \varphi \vert_{g}\,+\,\sup_{s\geq s_{0}}\int\limits_{\{\varphi=s\}}d\sigma_{g}<\,+\infty\,.
\end{align}
\end{lemma} 
\begin{proof}
Let $\psi$ be a chart at infinity. Considering $\widetilde{g}_{0}=\psi_{*}g_{0}=\widetilde{g}_{0;ij}dx^{i}\otimes dx^{j}$, by formulas~\eqref{eq37app} and~\eqref{eq12}, the coordinate expression of $$\vert \nabla \varphi \vert_{g}^{2}= \frac{4\,\vert \mathrm{D}u \vert^{2}}{(1-u^{2})^{2\frac{n-1}{n-2}}}$$
is
\begin{align}
\psi_{*}\vert \nabla \varphi \vert_{g}^{2}&= \frac{4\,\psi_{*}\vert \mathrm{D}u \vert^{2}}{(1-\widetilde{u}^{2})^{2\frac{n-1}{n-2}}}=\frac{4\,\widetilde{g}_{0}^{ij}\,\partial_{i}\widetilde{u}\,\partial_{j}\widetilde{u} }{\bigl\{1-[1-\mathcal{C}\vert x\vert^{2-n}+o(\vert x\vert^{2-n})]^{2}\bigl\}^{2\,(\frac{n-1}{n-2})}}\\
&=\frac{4\,\big[ \delta^{ij}+O\big(\vert x\vert^{-p}\big)\big]\,\big[(n-2)^{2}\,\mathcal{C}^{2}\,\vert x\vert^{-2n}\,x^{i}\,x^{j}+o(\vert x\vert^{2-2n})\big]}{\big[ 2\,\mathcal{C}\,\vert x\vert^{2-n}+o(\vert x\vert^{2-n})\big]^{2\,(\frac{n-1}{n-2})}}\\
&=\frac{4\,(n-2)^{2}\,\mathcal{C}^{2}\,\vert x\vert^{2-2n}+o(\vert x\vert^{2-2n})}{\big(2\,\mathcal{C}\big)^{2\,(\frac{n-1}{n-2})}\, \vert x\vert^{2-2n}\,\big(1+o(1)\big)}=\frac{4\,(n-2)^{2}\,\mathcal{C}^{2}+o(1)}{\big(2\,\mathcal{C}\big)^{2\,(\frac{n-1}{n-2})}\,\big(1+o(1)\big)}\,.
\end{align}
Hence 
\begin{equation}\label{eq60}
\vert \nabla \varphi \vert_{g}^{2}\longrightarrow \frac{4\,(n-2)^{2}\,\mathcal{C}^{2}}{\big(2\,\mathcal{C}\big)^{2\,(\frac{n-1}{n-2})}}=\big( 2\,\mathcal{C}\big)^{-\frac{2}{n-2}}(n-2)^{2}\,\,\text{at}\, \,\infty\,.
\end{equation}
Moreover, by limit~\eqref{eq60} there exist a constant $L>0$ and a value $s_{0}>0$ of $\varphi$ such that every $s\geq s_{0}$ is a regular value of $\varphi$ and $(1-u^{2})^{\frac{n-1}{n-2}}\leq L\,\,\vert \mathrm{D}u \vert$ on $\{\varphi\geq s_{0}\}$. Then
\begin{align}
\int\limits_{\{\varphi=s\}}d\sigma_{g}=\int\limits_{\{u=\tanh\frac{s}{2}\}}(1-u^{2})^{\frac{n-1}{n-2}}\,d\sigma\leq L\int\limits_{\{u=\tanh\frac{s}{2}\}}\vert \mathrm{D}u \vert\,d\sigma=L\int\limits_{\partial M}\vert \mathrm{D}u \vert\,d\sigma\,,
\end{align}
where in the last equality we have applied the Divergence Theorem. Consequently, we have that $\sup_{s\geq s_{0}}\int\limits_{\{\varphi=s\}}d\sigma_{g}< +\infty$. 
Similarly, we have that
\begin{align*}
\psi_{*}\vert \mathrm{D}^{2}u \vert^{2}\!&=\!\widetilde{g}_{0}^{\,i_{1}i_{2}}\widetilde{g}_{0}^{\,j_{1}j_{2}}(\mathrm{D}_{\widetilde{g}_{0}}^{2}\widetilde{u})_{i_{1}j_{1}}\!(\mathrm{D}_{\widetilde{g}_{0}}^{2}\widetilde{u})_{i_{2}j_{2}}\!=\!\widetilde{g}_{0}^{\,i_{1}i_{2}}\widetilde{g}_{0}^{\,j_{1}j_{2}}\big[\partial _{i_{1}}\partial_{j_{1}}\widetilde{u}-\Gamma_{\widetilde{g}_{0};i_{1}j_{1}}^{k_{1}}\partial _{k_{1}}\widetilde{u}\big]\!\big[\partial _{i_{2}}\partial_{j_{2}}\widetilde{u}-\Gamma_{\widetilde{g}_{0};i_{2}j_{2}}^{k_{2}}\partial _{k_{2}}\widetilde{u}\big]\\
&=[ \delta^{\,i_{1}i_{2}}\delta^{\,j_{1}j_{2}}+O\big(\vert x\vert^{-p}\big)\big]\,\big[\partial _{i_{1}}\partial_{j_{1}}\widetilde{u}-O(\vert x\vert^{-(p+n)}\big]\,
\big[\partial _{i_{2}}\partial_{j_{2}}\widetilde{u}-O(\vert x\vert^{-(p+n)}) \big]\\
&=\big[(n-1)(n-2)\mathcal{C}\big]^{2}\vert x\vert^{-2n}+o(\vert x\vert^{-2n})
\end{align*}
due to formulas~\eqref{eq37app},~\eqref{eq13} and~\eqref{eq46}. Moreover,
\begin{align}
\psi_{*}\mathrm{D}^{2}u(\mathrm{D}u,\mathrm{D}u)&=\mathrm{D}_{\widetilde{g}_{0}}^{2}\widetilde{u}(\mathrm{D}_{\widetilde{g}_{0}}\widetilde{u},\mathrm{D}_{\widetilde{g}_{0}}\widetilde{u})=\widetilde{g}_{0}^{\,ir}\,\widetilde{g}_{0}^{\,js}\,\partial_{r}\widetilde{u}\,\partial_{s}\widetilde{u}\,\big(\partial _{i}\partial_{j}\widetilde{u}-\Gamma_{\widetilde{g}_{0};ij}^{k}\partial _{k}\widetilde{u})\,\\
&=-(n-1)\,(n-2)^{3}\,\mathcal{C}^{3}\vert x\vert^{2-3n}+o(\vert x\vert^{2-3n})\,.
\end{align}
All in all,
\begin{align}
\psi_{*}\vert \nabla^{2} \varphi \vert_{g}^{2}&=\frac{4}{(1-\widetilde{u}^{2})^{\frac{2n}{n-2}}}\,\psi_{*}\vert \mathrm{D}^{2}\widetilde{u} \vert^{2}+\frac{16n}{n-2}\,\frac{\widetilde{u}}{(1-\widetilde{u}^{2})^{\frac{3n-2}{n-2}}}\,\psi_{*}\mathrm{D}^{2}\widetilde{u}(\mathrm{D}\widetilde{u},\mathrm{D}\widetilde{u})\\
&+\frac{16n(n-1)}{(n-2)^{2}}\,\frac{\widetilde{u}^{2}}{(1-\widetilde{u}^{2})^{4\big(\frac{n-1}{n-2}\big)}}\,\psi_{*}\vert \mathrm{D}\widetilde{u} \vert^{4}\\
&=(n-1)\,(n-2)^{2}\,(2\,\mathcal{C})^{-\frac{4}{n-2}}\,\Bigr\{\frac{n-1+o(1)}{1+o(1)}-\frac{2\,n\,\widetilde{u}+o(1)}{1+o(1)}
+\frac{16\,n\,\widetilde{u}^{2}+o(1)}{1+o(1)}\Bigr\}\,,
\end{align}
which gives
\begin{align}
\vert \nabla^{2} \varphi \vert_{g}^{2}\longrightarrow (n-1)\,(15n-1)\,(n-2)^{2}\,(2\,\mathcal{C})^{-\frac{4}{n-2}}\quad\,\,\text{at}\, \,\infty\,.
\end{align}
In particular, since $\varphi$ is smooth, we have that
$$\sup_{M}\vert \nabla \varphi \vert_{g}\,+\,\sup_{M}\vert \nabla^{2} \varphi \vert_{g}\,< \,+\infty\,.$$
\end{proof}

\begin{remark}
Note that
$\sup_{s\geq 0}\int\limits_{\{\varphi=s\}}d\sigma_{g}\in(0,+\infty]$,
since we cannot a priori exclude that there exist a critical value $\overline{s}>0$ and a sequence $\{s_{m}\}\subset (0,+\infty)$ such that $s_{m}\to\overline{s}$ and $$\int\limits_{\{\varphi=s_{m}\}}d\sigma_{g}\to +\infty\,.$$
\end{remark}

\noindent 
As it will be clear in the proof of the integral identity~\eqref{partialfirstintegralidentities2}, which is at the core of the conformal--monotonicity result (Proposition~\ref{Conforme Monotonicity theorem}), it is useful to introduce a suitable vector field with nonnegative divergence. To do this, let us focus on the set $\mathring{M}\setminus\mathrm{Crit}(\varphi)$ and notice first that the classical Bochner formula, applied to the $g$--harmonic function $\varphi$, becomes
\begin{align}
\frac{1}{2}\,\Delta_{g}\vert \nabla \varphi \vert^{2}_{g}\,=\,\vert \nabla^{2} \varphi \vert_{g}^{2}\,+\,\mathrm{Ric}_{g}(\nabla \varphi ,\nabla \varphi )+\,g(\nabla\Delta_{g}\varphi,\nabla \varphi )\,=\,\vert \nabla^{2} \varphi \vert_{g}^{2}\,+\,\mathrm{Ric}_{g}(\nabla \varphi ,\nabla \varphi )\,.\label{eq55}
\end{align}
Then, we obtain
\begin{align}
\Delta_{g}\vert \nabla \varphi \vert^{\beta}_{g}&=\mathrm{div}_{g}\Big(\nabla \vert \nabla \varphi \vert^{\beta}_{g}\Big)=\mathrm{div}_{g}\Big(\frac{\beta}{2}\vert \nabla \varphi \vert^{\beta-2}_{g}\nabla\vert \nabla \varphi \vert^{2}_{g}\Big)\\
&=\frac{\beta}{2}\Big[g(\nabla \vert \nabla \varphi \vert^{\beta-2}_{g},\nabla\vert \nabla \varphi \vert^{2}_{g})+ \vert \nabla \varphi \vert^{\beta-2}_{g}\Delta_{g}\vert \nabla \varphi \vert^{2}_{g} \Big]\\
&=\beta\,\vert\nabla \varphi \vert^{\beta-2}_{g}\Big[\,(\beta-2)\,\Big\vert\nabla\vert \nabla \varphi \vert_{g}\,\Big\vert_{g}^{2}+\vert \nabla^{2} \varphi \vert_{g}^{2}+\mathrm{Ric}_{g}(\nabla \varphi ,\nabla \varphi )\Big]\,, \label{eq56}
\end{align}
where in the third equality we have used~\eqref{eq55}. Now, observe from the nonnegativity of the tensor
\begin{equation}
\label{definizione_Q}
Q\,:=\,\mathrm{Ric}_{g}-\coth(\varphi)\,\nabla ^{2}\varphi+\frac{1}{n-2}\,d\varphi \otimes d\varphi-\frac{1}{n-2}\,\vert \nabla \varphi \vert^{2}_{g}\,\ g
\end{equation}
(see~\eqref{sistconf}) that 
\begin{align}\label{Q}
Q( \nabla \varphi, \nabla \varphi)=\mathrm{Ric}_{g}( \nabla \varphi, \nabla \varphi)-\coth(\varphi)\,\nabla ^{2}\varphi( \nabla \varphi, \nabla \varphi)\geq 0\,.
\end{align}
Therefore, by adding and subtracting the term $\beta\,\vert\nabla \varphi \vert^{\beta-2}_{g}\,\coth(\varphi)\,\nabla ^{2}\varphi( \nabla \varphi, \nabla \varphi)$
on the right--hand side of~\eqref{eq56}, we get
\begin{align}
\Delta_{g}\vert \nabla \varphi \vert^{\beta}_{g}-\beta\,\vert\nabla \varphi \vert^{\beta-2}_{g}&\,\coth(\varphi)\,\nabla ^{2}\varphi( \nabla \varphi, \nabla \varphi)\\
&=\beta\vert \nabla \varphi \vert^{\beta-2}_{g}\Big[\,(\beta-2)\Big\vert\nabla\vert \nabla \varphi \vert_{g}\,\Big\vert_{g}^{2}+\vert \nabla^{2} \varphi \vert_{g}^{2}+Q(\nabla \varphi ,\nabla \varphi )\,\Big]\,.\label{eq57}
\end{align}
Since
$$\beta\,\vert\nabla \varphi \vert^{\beta-2}_{g}\,\coth(\varphi)\,\nabla ^{2}\varphi( \nabla \varphi, \nabla \varphi)
=\coth(\varphi)\, g(\nabla\vert \nabla \varphi \vert_{g}^{\beta}, \nabla \varphi)$$
and since, setting
\begin{align}\label{eq61}
Y_{\beta}:=\frac{\nabla\, \vert \nabla \varphi \vert^{\beta}_{g}}{\sinh \varphi}\,, 
\end{align}
there holds
$$\mathrm{div}_{g}\,Y_{\beta}=\frac{\Delta_{g}\vert \nabla \varphi \vert^{\beta}_{g}}{\sinh \varphi}-\frac{\cosh \varphi}{\sinh^2 \varphi}\, \,g(\nabla\vert \nabla \varphi \vert_{g}^{\beta}, \nabla \varphi)\,,$$
from~\eqref{eq57} we get
\begin{align}
\sinh (\varphi)\, \mathrm{div}_{g}\,Y_{\beta}=\beta\vert \nabla \varphi \vert^{\beta-2}_{g}\Big[\,(\beta-2)\Big\vert\nabla\vert \nabla \varphi \vert_{g}\,\Big\vert_{g}^{2}+\vert \nabla^{2} \varphi \vert_{g}^{2}+Q(\nabla \varphi ,\nabla \varphi )\,\Big]\,.
\end{align}
Note that, by the refined Kato inequality for harmonic function 
\begin{align}\label{Katoineqharmfunct}
\frac{n}{n-1}\Big\vert\nabla\vert \nabla \varphi \vert_{g}\,\Big\vert_{g}^{2}\leq \vert \nabla^{2} \varphi \vert_{g}^{2}\,,
\end{align}
we have that
\begin{align}
(\beta-2)\Big\vert\nabla\vert \nabla \varphi \vert_{g}\,\Big\vert_{g}^{2}&+\vert \nabla^{2} \varphi \vert_{g}^{2}+Q(\nabla \varphi ,\nabla \varphi )\\
&=\Big(\beta-\frac{n-2}{n-1}\Big)\Big\vert\nabla\vert \nabla \varphi \vert_{g}\,\Big\vert_{g}^{2}+\Bigg[\vert \nabla^{2} \varphi \vert_{g}^{2}-\frac{n}{n-1}\Big\vert\nabla\vert \nabla \varphi \vert_{g}\,\Big\vert_{g}^{2}\Bigg]+Q(\nabla \varphi ,\nabla \varphi )\geq0\,,\,\label{eq65}
\end{align}
whenever $\beta\geq\frac{n-2}{n-1}$. Hence, $\mathrm{div}_{g}\,Y_{\beta}\geq 0$ for every $\beta\geq\frac{n-2}{n-1}$. 
This fact will be heavily used in the proof of the forthcoming results. It will also be useful to have a precise estimate of $\int\limits_{\{\vert \nabla \varphi \vert=\delta\}}\big\vert\,\nabla\vert \nabla \varphi \vert_{g}\,\,\big\vert_{g}\,d\sigma_{g}$ in terms of a suitable power of $\delta$, close to $\mathrm{Crit}(\varphi)$, that is when $\delta\to 0^{+}$. This is the content of the following lemma.

\begin{lemma}\label{lemmatecnico}
There exists $\delta_{0}>0$ such that 
\begin{align}\label{disutillemma}
\sup\Bigg\{\delta^{-\frac{1}{n-1}}\int\limits_{\{\vert \nabla \varphi \vert=\delta\}}\big\vert\,\nabla\vert \nabla \varphi \vert_{g}\,\,\big\vert_{g}\,d\sigma_{g}\,:\, 0<\delta<\delta_{0}\,\,\text{regular value of}\,\,\vert \nabla \varphi \vert_{g}\,\Bigg\}<+\infty\,.
\end{align}
\end{lemma}
\medskip

\noindent
We recall that the set of the critical values of $\vert \nabla \varphi \vert^{2}_{g}$ has zero Lebesgue measure by Sard's Theorem,
whereas we have no information regarding the 
local $\mathcal{H}$--dimension of 
$\mathrm{Crit}(\vert \nabla \varphi \vert^{2}_{g})$.
\begin{proof}
Applying Sard's Theorem to the smooth function $\vert \nabla \varphi \vert^{2}_{g}$ there exists $\varepsilon_{0}>0$ such that $\varepsilon_{0}$ is a regular value of $\vert \nabla \varphi \vert^{2}_{g}$ and 
$$\varepsilon_{0}<\min\Big\{\min_{\{\varphi=0\}}\vert \nabla \varphi \vert^{2}_{g},\text{ the limit of $\vert \nabla \varphi \vert^{2}_{g}$ at $\infty$ }\Big\}\,, 
$$
\noindent
where the limit in the previous expression is the (finite and positive) value computed in~\eqref{eq60}. 
In particular, $\{\vert \nabla \varphi \vert^{2}_{g}\leq \varepsilon_{0}\}$ is compactly contained in $\mathring{M}$, and for every $0<\delta<\delta_{0}$ regular value of $\vert \nabla \varphi \vert_{g}$ we have that
\begin{align}
\delta^{-\frac{1}{n-1}}\int\limits_{\{\vert \nabla \varphi \vert=\delta\}}\big\vert\,\nabla\vert \nabla \varphi \vert_{g}\,\,\big\vert_{g}\,d\sigma_{g}&=\int\limits_{\{\vert \nabla \varphi \vert=\delta\}}\sinh(\varphi)\,
 \frac{\vert \nabla \varphi \vert_{g}^{-\frac{1}{n-1}}\big\vert\,\nabla\vert \nabla \varphi \vert_{g}\,\,\big\vert_{g}}{\sinh(\varphi)}\,d\sigma_{g}\\
&\leq c\,\int\limits_{\{\vert \nabla \varphi \vert=\delta\}}\frac{\vert \nabla \varphi \vert_{g}^{-\frac{1}{n-1}}\big\vert\,\nabla\vert \nabla \varphi \vert_{g}\,\,\big\vert_{g}}{\sinh(\varphi)}\,d\sigma_{g}
\end{align}
Now, consider the smooth vector field
\begin{align}
Z:=2\,\frac{n-1}{n-2} \,Y_{\frac{n-2}{n-1}}= \frac{1}{\sinh \varphi}\,\, \frac{\nabla\vert \nabla \varphi \vert^{2}_{g}}{\quad\vert \nabla \varphi \vert^{\frac{n}{n-1}}_{g}}\,,
\end{align}
with $\mathrm{div}_{g}Z\geq0$.
Set 
\begin{align}\label{eq62}
U_{\mu}:=\{\vert \nabla \varphi \vert^{2}_{g}<\mu\}\quad \text{for every $\mu>0$}\,.
\end{align}
Then, for every $0<\varepsilon<\varepsilon_{0}$ regular value of the function $\vert \nabla \varphi \vert^{2}_{g}$\,, we apply the Divergence Theorem to the smooth vector field $Z$ on $U_{\varepsilon_{0}} \setminus \overline{U_{\varepsilon}}$, and we get
\begin{align}
\int\limits_{\{\vert \nabla \varphi \vert^{2}_{g}=\varepsilon_{0}\}}\frac{1}{\sinh \varphi}\,\,\frac{\quad\big\vert\,\nabla\vert \nabla \varphi \vert^{2}_{g}\,\,\big\vert_{g}}{\quad\vert \nabla \varphi \vert^{\frac{n}{n-1}}_{g}}\,d\sigma_{g}&-\int\limits_{\{\vert \nabla \varphi \vert^{2}_{g}=\varepsilon\}}\frac{1}{\sinh \varphi}\,\,\frac{\quad\big\vert\,\nabla\vert \nabla \varphi \vert^{2}_{g}\,\,\big\vert_{g}}{\quad\vert \nabla \varphi \vert^{\frac{n}{n-1}}_{g}}\,d\sigma_{g}\\
&=\int\limits_{U_{\varepsilon_{0}} \setminus\overline{U_{\varepsilon}}}\mathrm{div}_{g}Z\,d\mu_{g}\geq 0\,.
\end{align}
Then, it follows
\begin{align}
\int\limits_{\{\vert \nabla \varphi \vert^{2}_{g}=\varepsilon_{0}\}}\frac{1}{\sinh \varphi}\,\,\frac{\quad\big\vert\,\nabla\vert \nabla \varphi \vert^{2}_{g}\,\,\big\vert_{g}}{\quad\vert \nabla \varphi \vert^{\frac{n}{n-1}}_{g}}\,d\sigma_{g}\geq\int\limits_{\{\vert \nabla \varphi \vert^{2}_{g}=\varepsilon\}}\frac{1}{\sinh \varphi}\,\,\frac{\quad\big\vert\,\nabla\vert \nabla \varphi \vert^{2}_{g}\,\,\big\vert_{g}}{\quad\vert \nabla \varphi \vert^{\frac{n}{n-1}}_{g}}\,d\sigma_{g}\,.
\end{align}
Therefore, setting
$$c_{1}:=\int\limits_{\{\vert \nabla \varphi \vert^{2}_{g}=\varepsilon_{0}\}}\frac{1}{\sinh \varphi}\,\,\frac{\quad\big\vert\,\nabla\vert \nabla \varphi \vert^{2}_{g}\,\,\big\vert_{g}}{\quad\vert \nabla \varphi \vert^{\frac{n}{n-1}}_{g}}\,d\sigma_{g}>0\,,$$
we obtain 
\begin{align}\label{disutillemma1}
\frac{1}{\varepsilon^{\frac{1}{2}\frac{n}{n-1}}}\int\limits_{\{\vert \nabla \varphi \vert^{2}_{g}=\varepsilon\}}\frac{\quad\big\vert\,\nabla\vert \nabla \varphi \vert^{2}_{g}\,\,\big\vert_{g}}{\sinh \varphi}\,d\sigma_{g}\leq c_{1}\,.
\end{align}
Consequently, the desired statement follows keeping in mind that: if $\delta$ is a regular value of $\vert \nabla \varphi \vert_{g}$, then $\delta^{2}$ is a regular value of $\vert \nabla \varphi \vert_{g}^{2}$; in $M\setminus\mathrm{Crit}(\varphi)$ we have 
$\nabla\vert \nabla \varphi \vert^{2}_{g}=2\vert \nabla \varphi \vert_{g}\nabla\vert \nabla \varphi \vert_{g}$.
\end{proof}

\noindent
{\em We underline that from now on we will use Remark~\ref{remark2} widely}.
\smallskip

\noindent
The following proposition contains the integral identity which is the
main tool of our analysis.
\begin{proposition}\label{itegralidenty}
Let $(M,g_{0},u)$ be a sub-static harmonic triple, and let $g$ and $\varphi$ be the metric and the function defined in~\eqref{eq50}. 
Then, for every $\beta> \frac{n-2}{n-1}$ and for every $S>s>0$ regular values of $\varphi$, it holds
\begin{multline}
\label{partialfirstintegralidentities2}
\int\limits_{\{\varphi=s\}} \frac{\vert \nabla \varphi \vert^{\beta}_{g}\,\,\mathrm{H}_{g}}{\sinh\varphi}\,d\sigma_{g}-\int\limits_{\{\varphi=S\}} \frac{\vert \nabla \varphi \vert^{\beta}_{g}\,\,\mathrm{H}_{g}}{\sinh\varphi}\,d\sigma_{g}\\
=\int\limits_{\{s<\varphi<S\}}\frac{\vert \nabla \varphi \vert^{\beta-2}_{g}\,\Big[\,(\beta-2)\,\Big\vert\nabla\vert \nabla \varphi \vert_{g}\,\Big\vert_{g}^{2}+\vert \nabla^{2} \varphi \vert_{g}^{2}+Q(\nabla \varphi ,\nabla \varphi )\, \Big]}{\sinh\varphi}\,d\mu_{g}\,,
\end{multline}
where the tensor $Q$ is defined as in \eqref{definizione_Q}.
\end{proposition}

\begin{proof}
The case $\beta\geq 2$ is an easy adaptation of the argument used in~\cite{Virginia1} but it is anyway a consequence of the following argument. We focus on the unknown case $\frac{n-2}{n-1}<\beta<2$.
In $\mathring{M}\setminus\mathrm{Crit}(\varphi)$ we consider the smooth vector field $Y_{\beta}$, defined in~\eqref{eq61} and satisfying
\begin{align}\label{divY}
0\leq\mathrm{div}_{g}\,Y_{\beta}&=
\frac{\beta\,\vert \nabla \varphi \vert^{\beta-2}_{g}\,\Big[\,(\beta-2)\Big\vert\nabla\vert \nabla \varphi \vert_{g}\,\Big\vert_{g}^{2}+\vert \nabla^{2} \varphi \vert_{g}^{2}+Q(\nabla \varphi ,\nabla \varphi )\,\Big]}{\sinh \varphi}\,,
\end{align}
as already explained. Set
\begin{align}
E_{s}^{S}:=\{s<\varphi<S\},\quad\text{for every $S>s>0$}\,.\,\label{defEsS}
\end{align}
When $ E_{s}^{S}\cap \mathrm{Crit}(\varphi)=\emptyset$, then the statement is a straightforward application of the Divergence Theorem. Now, suppose that $ E_{s}^{S}\cap \mathrm{Crit}(\varphi)\neq\emptyset$.
In this case we consider, for every $\varepsilon>0$ sufficiently small, a smooth nondecreasing cut--off function $\chi_{\varepsilon}:[0,+\infty)\to [0,1]$ satisfying the following conditions
\begin{align}
\chi_{\varepsilon}(\tau)=0\,\text{in $\Big[\,0,\frac{1}{2}\varepsilon\,\Big]$}\,,\quad 0\leq\chi_{\varepsilon}'(\tau)\leq c\varepsilon^{-1}\,\text{in $\Big[\,\frac{1}{2}\varepsilon,\frac{3}{2}\varepsilon\,\Big]$}\,,\quad \chi_{\varepsilon}(\tau)=1\,\text{in $\Big[\,\frac{3}{2}\varepsilon,+\infty\,\Big)$}\,,\,\label{defcutoff}
\end{align}
where $c$ is a positive real constant independent of $\varepsilon$. We then define the smooth function $\Xi_{\varepsilon}:M\to [0,1]$ as
\begin{align}\label{Xieps}
\Xi_{\varepsilon}=\chi_{\varepsilon}\circ \vert \nabla \varphi \vert^{2}_{g}\,,
\end{align}
and apply the Divergence Theorem to the smooth vector field $\Xi_{\varepsilon}\, Y_{\beta}$ in $E_{s}^{S}$.
In this way, we get
\begin{align}
\int\limits_{\{\varphi=s\}} \frac{\vert \nabla \varphi \vert^{\beta}_{g}\,\,\mathrm{H}_{g}}{\sinh\varphi}\,d\sigma_{g}&-\int\limits_{\{\varphi=S\}} \frac{\vert \nabla \varphi \vert^{\beta}_{g}\,\,\mathrm{H}_{g}}{\sinh\varphi}\,d\sigma_{g}=\beta^{-1}\Bigg[\int\limits_{E_{s}^{S}}\Xi_{\varepsilon}\,\mathrm{div}_{g}\,Y_{\beta}\,d\mu_{g}+\int\limits_{E_{s}^{S}}g(\nabla\Xi_{\varepsilon},Y_{\beta})\,d\mu_{g}\Bigg]\\
&=\int\limits_{E_{s}^{S}}\frac{\,\Xi_{\varepsilon}\,\,\vert \nabla \varphi \vert^{\beta-2}_{g}\,\Big[\,(\beta-2)\Big\vert\nabla\vert \nabla \varphi \vert_{g}\,\Big\vert_{g}^{2}+\vert \nabla^{2} \varphi \vert_{g}^{2}+Q(\nabla \varphi ,\nabla \varphi )\,\Big]}{\sinh \varphi}\,d\mu_{g}\\
&+\int\limits_{ \big(U_{\frac{3}{2}\varepsilon} \setminus\overline{U_{\frac{1}{2}\varepsilon}}\big)\cap \, E_{s}^{S}}\,\frac{\chi'_{\varepsilon}\big(\,\vert \nabla \varphi \vert^{2}_{g}\,\big)\,\vert \nabla \varphi \vert^{\beta-2}_{g}\,\Big\vert\nabla\vert \nabla \varphi \vert^{2}_{g}\,\Big\vert_{g}^{2}}{2\sinh \varphi}\,d\mu_{g}\,,
\end{align}
where $U_{\mu}$ is defined in~\eqref{eq62}.
Note that $\{\chi_{\varepsilon}\}$ can always be chosen to be
nondecreasing in $\varepsilon$ so that, in turn, $\{\Xi_{\varepsilon}\}$ is nondecreasing.
Therefore, applying the Monotone Convergence Theorem, when $\varepsilon\to 0^+$, the first term on the right of the second equality tends to
\begin{align}
\int\limits_{E_{s}^{S}}\frac{\,\vert \nabla \varphi \vert^{\beta-2}_{g}\,\Big[\,(\beta-2)\Big\vert\nabla\vert \nabla \varphi \vert_{g}\,\Big\vert_{g}^{2}+\vert \nabla^{2} \varphi \vert_{g}^{2}+Q(\nabla \varphi ,\nabla \varphi )\,\Big]}{\sinh \varphi}\,d\mu_{g}\,.
\end{align}
For obtaining the desired statement, we show
\begin{align}\label{lim1}
\lim_{\varepsilon\to 0^{+}}\int\limits_{ \big(U_{\frac{3}{2}\varepsilon} \setminus\overline{U_{\frac{1}{2}\varepsilon}}\big)\cap \, E_{s}^{S}}\,\frac{\chi'_{\varepsilon}\big(\,\vert \nabla \varphi \vert^{2}_{g}\,\big)\,\vert \nabla \varphi \vert^{\beta-2}_{g}\,\Big\vert\nabla\vert \nabla \varphi \vert^{2}_{g}\,\Big\vert_{g}^{2}}{2\sinh \varphi}\,d\mu_{g}=0\,.
\end{align}
First we observe that
\begin{align*}
\int\limits_{ \big(U_{\frac{3}{2}\varepsilon} \setminus\overline{U_{\frac{1}{2}\varepsilon}}\big)\cap \, E_{s}^{S}}\,\frac{\chi'_{\varepsilon}\big(\,\vert \nabla \varphi \vert^{2}_{g}\,\big)\,\vert \nabla \varphi \vert^{\beta-2}_{g}\,\Big\vert\nabla\vert \nabla \varphi \vert^{2}_{g}\,\Big\vert_{g}^{2}}{2\sinh \varphi}\,d\mu_{g}
&\leq\,\,\int\limits_{ U_{\frac{3}{2}\varepsilon} \setminus\overline{U_{\frac{1}{2}\varepsilon}}}\,\frac{\chi'_{\varepsilon}\big(\,\vert \nabla \varphi \vert^{2}_{g}\,\big)\,\vert \nabla \varphi \vert^{\beta-2}_{g}\,\Big\vert\nabla\vert \nabla \varphi \vert^{2}_{g}\,\Big\vert_{g}^{2}}{2\sinh \varphi}\,d\mu_{g}\\
&\\
&\leq\frac{c}{2\varepsilon}\int\limits_{\frac{1}{2}\varepsilon}^{\frac{3}{2}\varepsilon}\,s^{\frac{\beta-2}{2}}ds\int\limits_{\{\vert \nabla \varphi \vert^{2}_{g}=s\}}\frac{\big\vert\,\nabla\vert \nabla \varphi \vert^{2}_{g}\,\,\big\vert_{g}}{\sinh \varphi}\,d\sigma_{g}
\end{align*}
where, keeping in mind the properties satisfied by $\chi_{\varepsilon}$, in the first inequality we have used the nonnegativity of the integrand function and in the last one the Coarea Formula.
Note that there exist $\varepsilon_{0},\,c_{1}>0$ such that the inequality
\begin{align}\label{disutil}
\frac{1}{s^{\frac{1}{2}\frac{n}{n-1}}}\int\limits_{\{\vert \nabla \varphi \vert^{2}_{g}=s\}}\frac{\quad\big\vert\,\nabla\vert \nabla \varphi \vert^{2}_{g}\,\,\big\vert_{g}}{\sinh \varphi}\,d\sigma_{g}\leq c_{1}
\end{align}
is true a.e. $s\in[\frac{1}{2}\varepsilon,\frac{3}{2}\varepsilon]$ for every $0<\varepsilon<\frac{2}{3}\varepsilon_{0}$, by both Sard's Theorem applied to the smooth function $\vert \nabla \varphi \vert^{2}_{g}\,$ and by Lemma~\ref{lemmatecnico}. 
Then, we get
\begin{align}
\int\limits_{ \big(U_{\frac{3}{2}\varepsilon} \setminus\overline{U_{\frac{1}{2}\varepsilon}}\big)\cap \, E_{s}^{S}}\,\frac{\chi'_{\varepsilon}\big(\,\vert \nabla \varphi \vert^{2}_{g}\,\big)\,\vert \nabla \varphi \vert^{\beta-2}_{g}\,\Big\vert\nabla\vert \nabla \varphi \vert^{2}_{g}\,\Big\vert_{g}^{2}}{2\sinh \varphi}\,d\mu_{g}
&\leq\frac{c}{2\varepsilon}\int\limits_{\frac{1}{2}\varepsilon}^{\frac{3}{2}\varepsilon}\,s^{\frac{\beta-2}{2}}ds\int\limits_{\{\vert \nabla \varphi \vert^{2}_{g}=s\}}\frac{\big\vert\,\nabla\vert \nabla \varphi \vert^{2}_{g}\,\,\big\vert_{g}}{\sinh \varphi}\,d\sigma_{g}\\
&\leq\frac{c\,c_{1}}{2\varepsilon}\,\int\limits_{\frac{1}{2}\varepsilon}^{\frac{3}{2}\varepsilon}\,s^{\frac{\beta-2}{2}+\frac{1}{2}\,\frac{n}{n-1}}\,ds\\
&=\frac{c\,c_{1}}{2\varepsilon}\,\int\limits_{\frac{1}{2}\varepsilon}^{\frac{3}{2}\varepsilon}\,s^{\frac{1}{2}(\beta-\frac{n-2}{n-1})}\,ds\\
&\leq c_{2}\,\varepsilon^{\frac{1}{2}(\beta-\frac{n-2}{n-1})}\,,\,\label{calutil1}
\end{align}
where $c_{2}>0$ is sufficiently big constant. This implies the limit in~\eqref{lim1}, because $\beta>\frac{n-2}{n-1}$.
\end{proof}

\begin{corollary} For every $S>s>0$ regular values of $\varphi$ there exists $r_{s,S}\geq0$ such that
\begin{align*}
\int\limits_{\{\varphi=s\}} \frac{\vert \nabla \varphi \vert^{\frac{n-2}{n-1}}_{g}\,\,\mathrm{H}_{g}}{\sinh\varphi}\,d\sigma_{g}&-\int\limits_{\{\varphi=S\}} \frac{\vert \nabla \varphi \vert^{\frac{n-2}{n-1}}_{g}\,\,\mathrm{H}_{g}}{\sinh\varphi}\,d\sigma_{g}\\
&=\int\limits_{\{s<\varphi<S\}}\frac{\vert \nabla \varphi \vert^{-\frac{n}{n-1}}_{g}\,\Big[\,\vert \nabla^{2} \varphi \vert_{g}^{2}-\frac{n}{n-1}\,\Big\vert\nabla\vert \nabla \varphi \vert_{g}\,\Big\vert_{g}^{2}+Q(\nabla \varphi ,\nabla \varphi )\, \Big]}{\sinh\varphi}\,d\mu_{g}+r_{s,S}\,.
\end{align*}
\end{corollary}
\begin{proof}
Let $\{\beta_{m}\}_{m\in \N}$ be a sequence such that $\beta_{m}>\frac{n-2}{n-1}$ and $\beta_{m}\to\frac{n-2}{n-1}$. Due to Proposition~\ref{itegralidenty}, we have
\begin{align}
\int\limits_{\{\varphi=s\}} \frac{\vert \nabla \varphi \vert^{\frac{n-2}{n-1}}_{g}\,\,\mathrm{H}_{g}}{\sinh\varphi}\,d\sigma_{g}&-\int\limits_{\{\varphi=S\}} \frac{\vert \nabla \varphi \vert^{\frac{n-2}{n-1}}_{g}\,\,\mathrm{H}_{g}}{\sinh\varphi}\,d\sigma_{g}\\
&=\lim_{m\to+\infty}
\Bigg[\int\limits_{\{\varphi=s\}} \frac{\vert \nabla \varphi \vert^{\beta_{m}}_{g}\,\,\mathrm{H}_{g}}{\sinh\varphi}\,d\sigma_{g}-\int\limits_{\{\varphi=S\}} \frac{\vert \nabla \varphi \vert^{\beta_{m}}_{g}\,\,\mathrm{H}_{g}}{\sinh\varphi}\,d\sigma_{g}\Bigg]\\
&=\lim_{m\to+\infty}\int\limits_{\{s<\varphi<S\}}\frac{\vert \nabla \varphi \vert^{\beta_{m}-2}_{g}\,\Big[\,(\beta_{m}-2)\,\Big\vert\nabla\vert \nabla \varphi \vert_{g}\,\Big\vert_{g}^{2}+\vert \nabla^{2} \varphi \vert_{g}^{2}+Q(\nabla \varphi ,\nabla \varphi )\, \Big]}{\sinh\varphi}\,d\mu_{g}\\
&\geq \int\limits_{\{s<\varphi<S\}}\frac{\vert \nabla \varphi \vert^{-\frac{n}{n-1}}_{g}\,\Big[\,\vert \nabla^{2} \varphi \vert_{g}^{2}-\frac{n}{n-1}\,\Big\vert\nabla\vert \nabla \varphi \vert_{g}\,\Big\vert_{g}^{2}+Q(\nabla \varphi ,\nabla \varphi )\, \Big]}{\sinh\varphi}\,d\mu_{g}\,,
\end{align}
where the first equality is consequence of the Dominate Converge Theorem keeping in mind that $s$ and $S$ are regular values of $\varphi$ while the inequality follows from Fatou's Lemma. Since $\{\beta_{m}\}_{m\in N}$ is arbitrary, the quantity
\begin{align}
r_{s,S}:=&\lim_{\beta\to\frac{n-2}{n-1}^+} \Bigg\{\,\int\limits_{\{s<\varphi<S\}}\frac{\vert \nabla \varphi \vert^{\beta-2}_{g}\,\Big[\,(\beta-2)\,\Big\vert\nabla\vert \nabla \varphi \vert_{g}\,\Big\vert_{g}^{2}+\vert \nabla^{2} \varphi \vert_{g}^{2}+Q(\nabla \varphi ,\nabla \varphi )\, \Big]}{\sinh\varphi}\,d\mu_{g}\Bigg\}\\
&\quad- \int\limits_{\{s<\varphi<S\}}\frac{\vert \nabla \varphi \vert^{-\frac{n}{n-1}}_{g}\,\Big[\,\vert \nabla^{2} \varphi \vert_{g}^{2}-\frac{n}{n-1}\,\Big\vert\nabla\vert \nabla \varphi \vert_{g}\,\Big\vert_{g}^{2}+Q(\nabla \varphi ,\nabla \varphi )\, \Big]}{\sinh\varphi}\,d\mu_{g}
\end{align}
is well--defined. Moreover, it is nonnegative as above and therefore we get the statement.
\end{proof}

\begin{remark}\label{consequenceofintegralidenty}
For every $\beta> \frac{n-2}{n-1}$ and for every $s>0$ regular value of the function $\varphi$:
\begin{align}
\int\limits_{\{\varphi=s\}} \frac{\vert \nabla \varphi \vert^{\beta}_{g}\,\,\mathrm{H}_{g}}{\sinh\varphi}\,d\sigma_{g}=
\int\limits_{\{\varphi>s\}}\frac{\vert \nabla \varphi \vert^{\beta-2}_{g}\,\Big[\,(\beta-2)\,\Big\vert\nabla\vert \nabla \varphi \vert_{g}\,\Big\vert_{g}^{2}+\vert \nabla^{2} \varphi \vert_{g}^{2}+Q(\nabla \varphi ,\nabla \varphi )\, \Big]}{\sinh\varphi}\,d\mu_{g}\,.\,\label{idintconsequenceofintegralidenty}
\end{align}
For every $S$ big enough, which is a regular value of $\varphi$, by Lemma~\ref{boundutili} with~\eqref{eq60} we have
$$\Bigg\vert \,\int\limits_{\{\varphi=S\}}\vert \nabla \varphi \vert^{\beta}_{g}\,\,\mathrm{H}_{g}\,d\sigma_{g}\,\Bigg\vert\leq \int\limits_{\{\varphi=S\}}\vert \nabla \varphi \vert^{\beta-1}_{g}\,\vert \nabla^{2} \varphi \vert_{g}\,d\sigma_{g}\,\leq \widetilde{c}\,.$$
In particular,
\begin{align}
\lim_{S\to +\infty}\,\frac{1}{\sinh(S)}\,\int\limits_{\{\varphi=S\}}\vert \nabla \varphi \vert^{\beta}_{g}\,\,\mathrm{H}_{g}\,d\sigma_{g}=\,0\,.
\end{align}
Therefore, the desired identity can be obtained by the Monotone Convergence Theorem, by passing to the limit as $S\to+\infty$ in~\eqref{partialfirstintegralidentities2}.
\end{remark}

\begin{remark}\label{sommintfunctsecondintegralidentities}
For every $\beta> \frac{n-2}{n-1}$, as consequence of integral identity~\eqref{partialfirstintegralidentities2}, we have 
\begin{align}\label{eq63}
\vert \nabla \varphi \vert^{\beta-2}_{g}\,\Big[\,(\beta-2)\,\Big\vert\nabla\vert \nabla \varphi \vert_{g}\,\Big\vert_{g}^{2}+\vert \nabla^{2} \varphi \vert_{g}^{2}+Q(\nabla \varphi ,\nabla \varphi )\, \Big]\,\in L^{1}_{\text{loc}}\big(\mathring{M},\mu_{g}\big)\,.
\end{align}
Since 
\begin{align*}
\int\limits_{K}\vert \nabla \varphi \vert^{\beta-3}_{g}\,\,\vert \nabla ^{2}\varphi(\nabla \varphi,\nabla \varphi)\vert\,d\mu_{g}\leq \int\limits_{K}\vert \nabla \varphi \vert^{\beta-1}_{g}\,\Big\vert\nabla\vert \nabla \varphi \vert_{g}\,\Big\vert_{g} \,d\mu_{g}=
\int\limits_{K}\vert \nabla \varphi \vert^{\frac{\beta}{2}}_{g}\,\vert \nabla \varphi \vert^{\frac{\beta-2}{2}}_{g}\,\Big\vert\nabla\vert \nabla \varphi \vert_{g}\,\Big\vert_{g} \,d\mu_{g}
\end{align*}
for every $K\subset \mathring{M}$ compact, by H\"{o}lder's Inequality from~\eqref{eq63} with~\eqref{eq65} we get that
\begin{align}\label{eq64}
\vert \nabla \varphi \vert^{\beta-3}_{g}\,\,\nabla ^{2}\varphi(\nabla \varphi,\nabla \varphi)\,\in L^{1}_{\text{loc}}\big(\mathring{M},\mu_{g}\big)\,.
\end{align}
\end{remark}
\medskip
 
\noindent
We need a final lemma before stating the (last and) most important result of this section. 
\begin{lemma}\label{firstitegralidenty}
Let $(M,g_{0},u)$ be a sub-static harmonic triple, and let $g$ and $\varphi$ be the metric and the function defined in~\eqref{eq50}. 
Then, the following statements hold true.
\begin{enumerate}[ label=(\roman*)]
\item For every $\beta\geq0$ and for every $S>s>0$:
\begin{align}
\int\limits_{\{\varphi=S\}}\frac{\vert \nabla \varphi \vert^{\beta+1}_{g}}{\sinh\varphi}\,d\sigma_{g}&-\int\limits_{\{\varphi=s\}}\frac{\vert \nabla \varphi \vert^{\beta+1}_{g}}{\sinh\varphi}\,d\sigma_{g}=\,\label{partialfirstintegralidentities1}\\
&=\int\limits_{\{s<\varphi<S\}}\frac{\vert \nabla \varphi \vert^{\beta-2}_{g}\,\Big[\,\beta\,\nabla ^{2}\varphi(\nabla \varphi,\nabla \varphi)\,-\coth(\varphi)\vert \nabla \varphi \vert^{4}_{g} \Big]}{\sinh\varphi}\,d\mu_{g}\,.
\end{align}
\item For every $\beta\geq0$ and for every $s>0$:
\begin{align}
\int\limits_{\{\varphi=s\}}\frac{\vert \nabla \varphi \vert^{\beta+1}_{g}}{\sinh\varphi}\,d\sigma_{g}=\int\limits_{\{\varphi>s\}}\frac{\vert \nabla \varphi \vert^{\beta-2}_{g}\,\Big[\,\coth(\varphi)\vert \nabla \varphi \vert^{4}_{g}-\beta\,\nabla ^{2}\varphi(\nabla \varphi,   \nabla \varphi)\, \Big]}{\sinh\varphi}\,d\mu_{g}\,.
\end{align}
\item The function $\Phi_{\beta}:[0,\infty)\to \R$, defined by formula
\begin{align}\label{Phibeta}
\Phi_{\beta}(s):=\int\limits_{\{\varphi=s\}}\vert \nabla \varphi \vert^{\beta+1}_{g}\,d\sigma_{g}
\end{align}
for every $\beta\geq0$, is continuous and admits for every $s>0$ the integral representation
\begin{align}
\Phi_{\beta}(s):=\sinh (s)\,\int\limits_{\{\varphi>s\}}\frac{\vert \nabla \varphi \vert^{\beta-2}_{g}\,\Big[\,\coth(\varphi)\vert \nabla \varphi \vert^{4}_{g}-\beta\,\nabla ^{2}\varphi(\nabla \varphi,\nabla \varphi)\, \Big]}{\sinh\varphi}\,d\mu_{g}\,.
\end{align}
\end{enumerate}
\end{lemma}

\noindent
This lemma can be proved as~\cite[Proposition 4.1]{Virginia1}. In the Appendix we provide an alternative proof which is self contained and does not make use of any fine property of the measure of $\mathrm{Crit}(\varphi)$: 
we just need to know very classical properties of it (see Remark~\ref{remark2}).

\begin{proposition}\label{Conforme Monotonicity theorem}
Let $(M,g_{0},u)$ be a sub-static harmonic triple, let $g$ and $\varphi$ be the metric and the function defined in~\eqref{eq50}, and let $\Phi_{\beta}:[0,\infty)\to \R$ be the function defined by formula~\eqref{Phibeta} for every $\beta\geq0$. Then for every $\beta>\frac{n-2}{n-1}$, the function $\Phi_{\beta}$ is continuously differentiable. The derivative $\Phi'_{\beta}$  is nonpositive and admits for every $s>0$ the integral representation
\begin{equation}\label{Phi'}
 \Phi'_{\beta}(s)=-\beta\, \sinh( s) \,\int\limits_{\{\varphi>s\}}\frac{\vert \nabla \varphi \vert^{\beta-2}_{g}\,\Big[\,(\beta-2)\,\Big\vert\nabla\vert \nabla \varphi \vert_{g}\,\Big\vert_{g}^{2}+\vert \nabla^{2} \varphi \vert_{g}^{2}+Q(\nabla \varphi ,\nabla \varphi )\, \Big]}{\sinh\varphi}\,d\mu_{g}\leq 0\,.
\end{equation}
Moreover, if there exists $s_{0}>0$ such that $\Phi'_{\beta}(s_{0})=0$ for some $\beta>\frac{n-2}{n-1}$, then $(\{\varphi\geq s_{0}\},g)$ is isometric to $\big([0,+\infty)\times \{\varphi=s_{0}\},d\rho\otimes d\rho+g_{\{\varphi=s_{0}\}})$, where $\rho$ is the $g$--distance function to $\{\varphi=s_{0}\}$ and $\varphi$ is an affine function of $\rho$ in $\{\varphi\geq s_{0}\}$.
\end{proposition}
\noindent
The following proof is essentially the same as in~\cite{Virginia1}. For completeness, we include it here, in a slightly refined version.
\begin{proof}{\em \underline{Step $1$: Continuous Differentiability and Monotonicity.}}
Let $\beta>\frac{n-2}{n-1}$. Note that the boundary $\partial M$ is a regular level set of $\varphi$ and then, by Theorem~\ref{geometryoflevelset} and the relationship between $\mathrm{Crit}(u)$ and $\mathrm{Crit}(\varphi)$, there exists $\epsilon_{0}$ such that the interval $[0,\epsilon_{0}]$ doesn't contain critical values of the function $\varphi$. Therefore, for every $0<\epsilon\leq\epsilon_{0}$, applying first the Divergence Theorem to the smooth vector field $\vert \nabla \varphi \vert^{\beta}_{g}\, \nabla \varphi$ in $\{0<\varphi<\epsilon\}$ and later the Coarea Formula, we get
$$\Phi_{\beta}(\epsilon)=\Phi_{\beta}(0)-\beta\,\int\limits_{0}^{\epsilon}\,ds\,\int\limits_{\{\varphi=s\}}\vert \nabla \varphi \vert^{\beta}_{g}\,\mathrm{H}_{g}\,d\sigma_{g}\,.$$
Being 
$$\int\limits_{\{\varphi=s_{1}\}}\vert \nabla \varphi \vert^{\beta}_{g}\,\mathrm{H}_{g}\,d\sigma_{g}-\int\limits_{\{\varphi=s_{2}\}}\vert \nabla \varphi \vert^{\beta}_{g}\,\mathrm{H}_{g}\,d\sigma_{g}=\beta^{-1}\int\limits_{\{s_{1}\leq\varphi\leq s_{2}\}}\mathrm{div}_{g}\big(\nabla \vert \nabla \varphi \vert^{\beta}_{g}\big)\,d\mu_{g}$$
for every $0\leq s_{1}<s_{2}\leq \epsilon_{0}$, by the Dominated Convergence Theorem the function $$s\in[0,\epsilon_{0}]\to\int\limits_{\{\varphi=s\}}\vert \nabla \varphi \vert^{\beta}_{g}\,\mathrm{H}_{g}\,d\sigma_{g}\in\R$$
is continuous and therefore, by the Fundamental Theorem of Calculus $\Phi_{\beta}$ is continuously differentiable on the closed interval $[0,\epsilon_{0}]$.\\
Let $s_{0}$ be a regular value of the function $\varphi$. By Remark~\ref{sommintfunctsecondintegralidentities}, we can define the function $\Psi_{\beta}:(0,+\infty)\to\R$ by
\begin{align}
\Psi_{\beta}(s)=\begin{cases}\int\limits_{\{\varphi=s_{0}\}} \frac{\vert \nabla \varphi \vert^{\beta}_{g}\,\,\mathrm{H}_{g}}{\sinh\varphi}\,d\sigma_{g}+
\int\limits_{\{s<\varphi<s_{0}\}}\frac{\vert \nabla \varphi \vert^{\beta-2}_{g}\,\Big[\,(\beta-2)\,\big\vert\nabla\vert \nabla \varphi \vert_{g}\,\big\vert_{g}^{2}+\vert \nabla^{2} \varphi \vert_{g}^{2}+Q(\nabla \varphi ,\nabla \varphi )\, \Big]}{\sinh\varphi}\,d\mu_{g}
&\mbox{if}\,s\leq s_{0}\\ \\
\int\limits_{\{\varphi=s_{0}\}} \frac{\vert \nabla \varphi \vert^{\beta}_{g}\,\,\mathrm{H}_{g}}{\sinh\varphi}\,d\sigma_{g}-
\int\limits_{\{s_{0}<\varphi<s\}}\frac{\vert \nabla \varphi \vert^{\beta-2}_{g}\,\Big[\,(\beta-2)\,\big\vert\nabla\vert \nabla \varphi \vert_{g}\,\big\vert_{g}^{2}+\vert \nabla^{2} \varphi \vert_{g}^{2}+Q(\nabla \varphi ,\nabla \varphi )\, \Big]}{\sinh\varphi}\,d\mu_{g}
&\mbox{if}\,s>s_{0} \,,
\end{cases}
\end{align}
which satisfies the following properties
\begin{enumerate}[label =(\em{\roman*})]
\item for every $s>0$ regular value of the function $\varphi$, we have $\Psi_{\beta}(s)=\int\limits_{\{\varphi=s\}} \frac{\vert \nabla \varphi \vert^{\beta}_{g}\,\,\mathrm{H}_{g}}{\sinh\varphi}\,d\sigma_{g}$\,;
\item the function $\Psi_{\beta}$ is continuous on its definition interval $(0,+\infty)$.
\end{enumerate}
The first statement follows immediately from Proposition~\ref{itegralidenty}. As for the second statement, we first observe that
\begin{align}\label{eq3}
\Psi_{\beta}(s)-\Psi_{\beta}(\overline{s})=\int\limits_{\{s<\varphi<\overline{s}\}}\frac{\vert \nabla \varphi \vert^{\beta-2}_{g}\,\Big[\,(\beta-2)\,\big\vert\nabla\vert \nabla \varphi \vert_{g}\,\big\vert_{g}^{2}+\vert \nabla^{2} \varphi \vert_{g}^{2}+Q(\nabla \varphi ,\nabla \varphi )\, \Big]}{\sinh\varphi}\,d\mu_{g}
\end{align}
for every couple $0<s<\overline{s}<+\infty$. Always by Remark~\ref{sommintfunctsecondintegralidentities} and by the Dominated Convergence Theorem, we can deduce the right and the left continuity of $\Psi_{\beta}$ on the interval $(0,+\infty)$.\\
We consider $\Upsilon_{\beta}:s\in(0,+\infty)\to\frac{\Phi_{\beta}(s)}{\sinh s}\in\R\,$.
For every $(s,\overline{s})$ couple of real number such that $0<s<\overline{s}<+\infty$, we have 
\begin{align}
\frac{\Upsilon_{\beta}(\overline{s})-\Upsilon_{\beta}(s)}{\overline{s}-s}&=\frac{1}{\overline{s}-s}\int\limits_{\{s<\varphi<\overline{s}\}}\frac{\vert \nabla \varphi \vert^{\beta-2}_{g}\,\Big[\,\beta\,\nabla ^{2}\varphi(\nabla \varphi,\nabla \varphi)\,-\coth(\varphi)\vert \nabla \varphi \vert^{4}_{g} \Big]}{\sinh\varphi}\,d\mu_{g}\\
&=\frac{1}{\overline{s}-s}\int\limits_{s}^{\overline{s}}d\tau\int\limits_{\{\varphi=\tau\}}\frac{\vert \nabla \varphi \vert^{\beta-3}_{g}\,\Big[\,\beta\,\nabla ^{2}\varphi(\nabla \varphi,\nabla \varphi)\,-\coth(\varphi)\vert \nabla \varphi \vert^{4}_{g} \Big]}{\sinh\varphi}\,d\sigma_{g}\\
&\overset{(\star)}{=} -\frac{\beta}{\overline{s}-s}\,\int\limits_{s}^{\overline{s}}\Psi_{\beta}(\tau)\,d\tau-\frac{1}{\overline{s}-s}\,\int\limits_{s}^{\overline{s}}\coth(\tau)\Upsilon_{\beta}(\tau)\, d\tau\,,
\end{align}
where the first equality follows from Lemma~\ref{firstitegralidenty} $(i)$, the second equality from the Coarea Formula keeping in mind~\eqref{eq64}. Moreover, the last equality follows from $(i)$ and from Sard's Theorem.
Using the continuity of both the functions $\Upsilon_{\beta}$ and $\Psi_{\beta}$, passing to the limit in $(\star)$ for either $s\to \overline{s}$ or $\overline{s}\to s$ yields that the function $\Upsilon_{\beta}$ is $C^{1}$, and 
$$\Upsilon_{\beta}'(\,\cdot\, )=-\beta\, \Psi_{\beta}(\,\cdot\,)-\coth(\,\cdot\,)\,\Upsilon_{\beta}(\,\cdot\,)\,.$$
Since $\Phi_{\beta}(s)=\sinh(s)\Upsilon_{\beta}(s)$ for every $s>0$, then $\Phi_{\beta}\in C^{1}(0,+\infty)$ and $\Phi'_{\beta}(s)=-\beta\,\sinh(s)\Psi_{\beta}(s)$. Moreover, by~\eqref{eq3}, we can see
\begin{align}\label{eq4}
\frac{\Phi_{\beta}'(S)}{\sinh(S)}-\frac{\Phi_{\beta}'(s)}{\sinh(s)}&=-\beta\,\Psi_{\beta}(S)+\beta\,\Psi_{\beta}(s)\\
&=\beta\,\int\limits_{\{s<\varphi<S\}}\frac{\vert \nabla \varphi \vert^{\beta-2}_{g}\,\Big[\,(\beta-2)\,\big\vert\nabla\vert \nabla \varphi \vert_{g}\,\big\vert_{g}^{2}+\vert \nabla^{2} \varphi \vert_{g}^{2}+Q(\nabla \varphi ,\nabla \varphi )\, \Big]}{\sinh\varphi}\,d\mu_{g}
\end{align}
for every $0<s<S<+\infty$.\\
Finally the integral representation~\eqref{Phi'} follows in the limit as $S\to +\infty$ of the above identity, by using the Monotone Convergence Theorem, and by the fact that 
$$\lim_{S\to+\infty}\,\frac{\Phi_{\beta}'(S)}{\sinh(S)}\,=-\beta \lim_{S\to+\infty}\,\Psi_{\beta}(S)=-\beta \lim_{S\to+\infty}\,\int\limits_{\{\varphi=S\}} \frac{\vert \nabla \varphi \vert^{\beta}_{g}\,\,\mathrm{H}_{g}}{\sinh\varphi}\,d\sigma_{g}=0\,.$$\\

\noindent
{\em \underline{Step $2$: Outer Rigidity.}}
Let $\beta>\frac{n-2}{n-1}$ and suppose $\Phi'_{\beta}(s_{0})=0$ for some $s_{0}>0$. By~\eqref{Phi'} with~\eqref{eq65} 
we deduce that 
\begin{equation}\label{eq14}
\Big(\beta-\frac{n-2}{n-1}\Big)\,\Big\vert\nabla\vert \nabla \varphi \vert_{g}\,\Big\vert_{g}^{2}\equiv0 \quad\text{and} \quad \vert \nabla^{2} \varphi \vert_{g}^{2}-\frac{n}{n-1}\Big\vert\nabla\vert \nabla \varphi \vert_{g}\,\Big\vert_{g}^{2}\equiv0\quad \text{in}\,\{\varphi\geq s_{0}\}\setminus\mathrm{Crit}(\varphi)\,.
\end{equation}
Consequently $\nabla^{2} \varphi\equiv 0$ in $\{\varphi\geq s_{0}\}$ being $\mu_{g}\big(\mathrm{Crit}(\varphi)\,\big)=0$, and hence $\vert \nabla \varphi \vert_{g}^{2}\equiv a^{2}$ with $a>0$ since $\{\varphi\geq s_{0}\}$ is connected,
due to Remark~\ref{remark2}.
Then, $\{\varphi\geq s_{0}\}$, with the induced Riemanninan metric, 
is a noncompact, connected and complete Riemannian manifold 
(being properly embedded in $M$), with smooth, compact and totally geodesic boundary, 
and with $\mathrm{Ric}_{g}\geq0$ (from the inequality in~\eqref{sistconf}).
Applying~\cite[Theorem C]{Kasue}, we can thus deduce that the level set 
$\{\varphi \,=\,s_{0}\}$ is connected (this is true in general and not 
only in the rigid case, if $s_0\gg0$, as observed in Remark~\ref{remark2}), and that
$\{\varphi\geq s_{0}\}$ is isometric to the product
$[0,+ \infty)\times\{\varphi \,=\,s_{0}\}$. Moreover, the isometry from the product $[0,+ \infty)\times\{\varphi \,=\,s_{0}\}$ to $\{\varphi\geq s_{0}\}$ is given by the normal exponential map.\\ 
Now we want to prove that $\varphi$ is an affine function of $\rho$ on $\{\varphi\geq s_{0}\}$.
First, we remark that every integral curve $\gamma_{p}$ of $\nabla \varphi$ outgoing from a point $p$ of $\{\varphi \,=\,s_{0}\}$ is defined on the interval $[0,+\infty)$, and it is contained in $\{\varphi\geq s_{0}\}$, by the completeness and since $\vert\nabla \varphi\vert_{g}>0$.
Furthermore, $\varphi \circ \gamma_{p}(t)=a^{2}t+s_{0}$ for every $t\in[0,+\infty)$, and all the curves $\gamma_{p}$ realize the distance between the hypersurfaces $\{\varphi = s_{0}\}$ and $\{\varphi = s_{1}\}$ with $s_{1}>s_{0}$.
Indeed, for any curve $\sigma:[0,l]\to \{\varphi\geq s_{0}\}$ parametrized by arc--length joining a point of $\{\varphi = s_{0}\}$ to a point of $\{\varphi = s_{1}\}$ we have
\begin{align}
L_{g}(\, \sigma\,)&=\int\limits_{0}^{l} \vert \,\dot{\sigma}(\tau)\,\vert_{g}\,\,d\tau\geq \Bigg\vert \int\limits_{0}^{l} g\Big(\,\dot{\sigma}(\tau),\frac{1}{a}\nabla \varphi \,\big(\sigma(t)\,\big)\,\Big)\,\,d\tau \Bigg\vert=\frac{1}{a}\vert \varphi \circ \sigma\,(l)-\varphi \circ \sigma(0)\vert \\
&=at=L_{g}(\, \gamma_{\sigma(0)}\,|_{[0,t]}\,)=L_{g}(\, \gamma_{\cdot}\,|_{[0,t]}\,)\,,
\end{align}
where $s_{1},s_{0}$ and $t$ satisfy $s_{1}=a^{2}t+s_{0}$.
Since $\xi=\frac{\,\nabla \varphi }{a}$ is the unit inner normal vector field of the boundary $\{\varphi \,=\,s_{0}\}$ and we just know that the normal exponential map is a diffeomorphism, $\exp^{\bot}(t\xi_{p})$ is a point having distance from $\{\varphi = s_{0}\}$ equal to $t\,$, and 
therefore $$\varphi\big(\exp^{\bot}(t\xi_{p})\,\big)=\varphi \circ \gamma_{p}\Big(\,\frac{t}{a}\,\Big)=at+s_{0}=a\, \rho\big(\exp^{\bot}(t\xi_{p})\big)+s_{0}\,.$$
This tell us that $\varphi$ is an affine function of $\rho$ on $\{\varphi\geq s_{0}\}$. 
\end{proof}


\noindent While the previous proposition contains an outer rigidity result, 
with the following corollary we provide a global rigidity result.
\begin{corollary}\label{corConforme Monotonicity theorem}
Let $(M,g_{0},u)$ be a sub-static harmonic triple, let $g$ and $\varphi$ be the metric and the function defined in~\eqref{eq50}, and let $\Phi_{\beta}:[0,\infty)\to \R$ be the function defined by formula~\eqref{Phibeta} for every $\beta\geq0$. If $\Phi_{\beta}$ is constant for some $\beta>\frac{n-2}{n-1}$, then $\partial M$ is connected and $(M,g)$ is isometric to $\big([0,+\infty)\times \partial M,d\rho\otimes d\rho+g_{\partial M})$, where $\rho$ is the $g$--distance function to $\partial M$ and $\varphi$ is an affine function of $\rho$.
\end{corollary}
\begin{proof}
We obtain immediately that $\Phi_{\beta}'(s)=0$ for every $s>0$. Thus, by formula~\eqref{Phi'} with~\eqref{eq65} we have that
$$\int\limits_{\{\varphi>s\}}\Bigg\{\Big(\beta-\frac{n-2}{n-1}\Big)\,\big\vert\nabla\vert \nabla \varphi \vert_{g}\,\big\vert_{g}^{2}+\Bigg[\,\vert \nabla^{2} \varphi \vert_{g}^{2}-\frac{n}{n-1}\Big\vert\nabla\vert \nabla \varphi \vert_{g}\,\Big\vert_{g}^{2}\,\Bigg]+Q(\nabla \varphi ,\nabla \varphi )\Bigg\}\,d\mu_{g}= 0$$
for every $s>0$. In turn, by the Monotone Convergence Theorem, we get 
$$\int\limits_{M}\Bigg\{\Big(\beta-\frac{n-2}{n-1}\Big)\,\big\vert\nabla\vert \nabla \varphi \vert_{g}\,\big\vert_{g}^{2}+\Bigg[\,\vert \nabla^{2} \varphi \vert_{g}^{2}-\frac{n}{n-1}\Big\vert\nabla\vert \nabla \varphi \vert_{g}\,\Big\vert_{g}^{2}\,\Bigg]+Q(\nabla \varphi ,\nabla \varphi )\Bigg\}\,d\mu_{g}= 0\,.$$
Then, we deduce that
$$\Big(\beta-\frac{n-2}{n-1}\Big)\,\Big\vert\nabla\vert \nabla \varphi \vert_{g}\,\Big\vert_{g}^{2}\equiv0 \quad\text{and} \quad \vert \nabla^{2} \varphi \vert_{g}^{2}-\frac{n}{n-1}\Big\vert\nabla\vert \nabla \varphi \vert_{g}\,\Big\vert_{g}^{2}\equiv0\quad \text{in}\, M\setminus\mathrm{Crit}(\varphi)\,,$$
due to Kato Inequality for harmonic functions~\eqref{Katoineqharmfunct} and by~\eqref{Q}.
Consequently $\nabla^{2} \varphi\equiv 0$ in $M$. The very same argument of the proof of Outer Rigidity in Proposition~\ref{Conforme Monotonicity theorem} implies that
$\partial M$ is connected and $(M,g)$ is isometric to 
$$\big([0,+\infty)\times \partial M,d\rho\otimes d\rho+g_{\partial M})$$
where $\rho$ is the $g$--distance to $\partial M$ and $\varphi$ is an affine function of $\rho$.
\end{proof}

\section{A Black--Hole uniqueness theorem for sub-static manifolds}
\label{sec_uniqueness}
This section is devoted to the proof of the Black--Hole uniqueness result for a sub--static harmonic triple, Theorem~\ref{uniq}.
We first recall the classical definition of ADM mass, together with an alternative characterization of it.
\medskip

\noindent
Let $(N,h)$ and $\psi$ be an asymptotically flat manifold with one end and a chart at infinity of $N$, respectively. We consider $\widetilde{h}:=\psi_{*}h=\widetilde{h}_{ij}dx^{i}\otimes dx^{j}$ and we set
\begin{align}
m(r)&=\frac{1}{2(n-1)\vert \SSS^{n-1}\vert}\int\limits_{\partial B_{r}}(\partial_{j}\widetilde{h}_{ij}-\partial_{i}\widetilde{h}_{jj})\nu_{e}^{i}d\sigma_{e}\,,\\
m_{I}(r)&=-\frac{1}{(n-2)(n-1)\vert \SSS^{n-1}\vert}\int\limits_{\partial B_{r}}\Big(\mathrm{Ric}_{\widetilde{h}}-\frac{1}{2}\mathrm{R}_{\widetilde{h}}\,\widetilde{h}\Big)(X,\nu_{\widetilde{h}})\,d\sigma_{\widetilde{h}}\,,
\end{align}
where 
$\nu_{e}$ and $\sigma_{e}$ are the $\infty$--pointing unit normal and the canonical measure on $\partial B_{r}$ as Riemannian submanifold of $(\R^{n}\setminus \overline{B},g_{\R^{n}})$, respectively,
and $\nu_{\widetilde{h}}$ and $\sigma_{\widetilde{h}}$ are the $\infty$--pointing unit normal and the canonical 
measure on $\partial B_{r}$ as Riemannian submanifold of $(\R^{n}\setminus \overline{B},\widetilde{h})$, respectively. Also,
$\mathrm{Ric}_{\widetilde{h}}$ and $\mathrm{R}_{\widetilde{h}}$ are the Ricci tensor and the scalar curvature of $\widetilde{h}$ respectively, and $X$ is the Euclidean conformal Killing vector field $x^{i}\,\frac{\partial }{\partial x^{i}}$.
The ADM mass is well defined as
\begin{equation}
m_{\mathrm{ADM}}:=\lim_{r\to +\infty}\, m(r)\,,
\end{equation}
and independent of the chosen chart at infinity. Moreover (see~\cite{miaotam}), it can be equivalently expressed as
\begin{equation}\label{ADMmasswithricc}
m_{\mathrm{ADM}}=\lim_{r\to +\infty}\, m_{I}(r)\,.
\end{equation}
\medskip

\noindent
From the alternative definition of ADM mass, given by~\eqref{ADMmasswithricc}, and using the Positive Mass Theorem, more precisely a consequence of it contained in~\cite[Theorem 1.5]{hirsch}, one can prove the following uniqueness statement. 
For the notation and terminology, we refer the reader to Definition~\ref{asymptoticallyflatmanifold} and Section~\ref{prel}.
\medskip

\noindent
\begin{proof}[\underline{Proof of Theorem~\ref{uniq}}.]
By condition~\eqref{eq58} and by the fact that $\mathrm{D}_{g_{0}}^{2}u\equiv 0$ on $\partial M$, which in turn implies $\mathrm{H}_{\partial M}^{g_{0}}\equiv 0$, we have that the hypothesis of~\cite[Theorem 1.5]{hirsch} are fulfilled, so that
\begin{equation}\label{eq36}
m_{\mathrm{ADM}}\geq \mathcal{C}.
\end{equation}
Now, we want to show that the reverse inequality holds. Let $\psi$ be a chart at infinity of $M$ (according to Definition~\ref{asymptoticallyflatmanifold}) and consider $\widetilde{g}_{0}=\psi_{*}{g_{0}}$.
Recalling that $\widetilde{u}$ stands for $u\circ\psi^{-1}$, we rewrite characterization~\eqref{ADMmasswithricc} as 
\begin{align*}
m_{\mathrm{ADM}}&=\lim_{r\to +\infty}\Bigg\{-\frac{1}{(n-2)(n-1)\vert \SSS^{n-1}\vert}\int\limits_{\partial B_{r}}\Big(\mathrm{Ric}_{\widetilde{g}_{0}}-\frac{\mathrm{D}_{\widetilde{g}_{0}}^{2}\widetilde{u}}{\widetilde{u}} \Big)(X,\nu_{e})\,d\sigma_{\widetilde{g}_{0}}\\
&-\frac{1}{(n-2)(n-1)\vert \SSS^{n-1}\vert}\int\limits_{\partial B_{r}}\Big(\mathrm{Ric}_{\widetilde{g}_{0}}-\frac{\mathrm{D}_{\widetilde{g}_{0}}^{2}\widetilde{u}}{\widetilde{u}} \Big)(X,\nu_{\widetilde{g}_{0}}-\nu_{e})\,d\sigma_{\widetilde{g}_{0}}\\
&- \frac{1}{(n-2)(n-1)\vert \SSS^{n-1}\vert}\int\limits_{\partial B_{r}} \frac{\mathrm{D}_{\widetilde{g}_{0}}^{2}\widetilde{u}}{\widetilde{u}}(X,\nu_{\widetilde{g}_{0}})\,d\sigma_{\widetilde{g}_{0}}+\frac{1}{2(n-2)(n-1)\vert \SSS^{n-1}\vert}\int\limits_{\partial B_{r}}\mathrm{R}_{\widetilde{g}_{0}}\,
\widetilde g_0(X,\nu_{\widetilde{g}_{0}})\,d\sigma_{\widetilde{g}_{0}}\Bigg\}\,.
\end{align*}
We note first that since $\nu_{e}=\frac{x^{i}}{\vert x\vert}\,\frac{\partial }{\partial x^{i}}=\frac{1}{\vert x\vert}\, X$ and $\widetilde{u}\,\mathrm{Ric}_{\widetilde{g}_{0}}-\mathrm{D}_{\widetilde{g}_{0}}^{2}\widetilde{u} \geq 0$ from the first equation in
~\eqref{f0}, we have
\begin{align}\label{eq44}
\int\limits_{\partial B_{r}}\Big(\mathrm{Ric}_{\widetilde{g}_{0}}-\frac{\mathrm{D}_{\widetilde{g}_{0}}^{2}\widetilde{u}}{\widetilde{u}} \Big)(X,\nu_{e})\,d\sigma_{\widetilde{g}_{0}}=\frac{1}{r}\int\limits_{\partial B_{r}}\Big(\mathrm{Ric}_{\widetilde{g}_{0}}-\frac{\mathrm{D}_{\widetilde{g}_{0}}^{2}\widetilde{u}}{\widetilde{u}} \Big)(X,X)\,d\sigma_{\widetilde{g}_{0}}\geq0\,.
\end{align}
Secondly, recalling that $(\mathrm{D}_{\widetilde{g}_{0}}^{2}\widetilde{u})_{ij}=\partial_{i}\partial_{j}\widetilde{u}-\Gamma_{ij}^{k}\partial_{k}\widetilde{u}$, where $\Gamma_{ij}^{k}$ are the the Christoffel symbols related to $\widetilde{g}_{0}$, and using~\eqref{eq6},~\eqref{simbchristoffel}, and the asymptotic expansions of $\widetilde{u}$, we get 
\begin{align}
\vert (\mathrm{D}_{\widetilde{g}_{0}}^{2}\widetilde{u})_{ij}-(\mathrm{D}_{e}^{2}\widetilde{u})_{ij}\vert&=\vert \Gamma_{ij}^{k}\partial_{k}\widetilde{u}\vert=O\big(\vert x\vert^{-(n+p)}\big)\label{eq46}\\
(\mathrm{D}_{\widetilde{g}_{0}}^{2}\widetilde{u})_{ij}&=O(\vert x\vert^{-n})\label{eq45}\,.
\end{align}
Decay~\eqref{eq45} coupled with~\eqref{eq39app}~\eqref{eq42app} and~\eqref{eq41app}
yields
\begin{align}
\Big\vert \int\limits_{\partial B_{r}}\Big(\mathrm{Ric}_{\widetilde{g}_{0}}-\frac{\mathrm{D}_{\widetilde{g}_{0}}^{2}\widetilde{u}}{\widetilde{u}}\,\Big)(X,\nu_{\widetilde{g}_{0}}-\nu_{e})\,d\sigma_{\widetilde{g}_{0}}\Big\vert \leq C \int\limits_{\partial B_{r}}\frac{1}{\vert x\vert ^{p+\min\{p+2,n\}-1}}\,d\sigma_{e}=\frac{C}{r^{p+\min\{p+2,n\}-n}}\xrightarrow {}0 \,,\label{eq70}
\end{align}
being $p>\frac{n-2}{2}$. Thirdly, we observe that
\begin{align}\label{eq47}
\int\limits_{\partial B_{r}} \frac{\mathrm{D}_{\widetilde{g}_{0}}^{2}\widetilde{u}}{\widetilde{u}}(X,\nu_{\widetilde{g}_{0}})\,d\sigma_{\widetilde{g}_{0}}\xrightarrow[r\to +\infty]{}-(n-1)(n-2)\,\mathcal{C}\vert \SSS^{n-1}\vert\,.
\end{align}
Indeed
\begin{align*}
\int\limits_{\partial B_{r}} \frac{\mathrm{D}_{\widetilde{g}_{0}}^{2}\widetilde{u}}{\widetilde{u}}(X,\nu_{\widetilde{g}_{0}})\,d\sigma_{\widetilde{g}_{0}}&=\int\limits_{\partial B_{r}} \frac{\mathrm{D}_{\widetilde{g}_{0}}^{2}\widetilde{u}}{\widetilde{u}}(X,\nu_{\widetilde{g}_{0}}-\nu_{e})\,d\sigma_{\widetilde{g}_{0}}+\int\limits_{\partial B_{r}} \frac{\mathrm{D}_{\widetilde{g}_{0}}^{2}\widetilde{u}-\mathrm{D}_{e}^{2}\widetilde{u}}{\widetilde{u}}(X,\nu_{e})\,d\sigma_{\widetilde{g}_{0}}\\
&+\int\limits_{\partial B_{r}} \frac{\mathrm{D}_{e}^{2}\widetilde{u}}{\widetilde{u}}(X,\nu_{e})\,d\sigma_{\widetilde{g}_{0}}\,,
\end{align*}
and one can show, with similar estimates as before, that the first two terms of this sum tend to $0$ for $r\to+\infty$.
It is also easy to see, using~\eqref{eq13} and~\eqref{eq41app}, that
$$\int\limits_{\partial B_{r}} \frac{\mathrm{D}_{e}^{2}\widetilde{u}}{\widetilde{u}}(X,\nu_{e})\,d\sigma_{\widetilde{g}_{0}}\xrightarrow[r\to +\infty]{}-(n-1)(n-2)\,\mathcal{C}\vert \SSS^{n-1}\vert\,.$$
Hence,~\eqref{eq47} is proven.
Gathering~\eqref{eq44} ,~\eqref{eq70}, and~\eqref{eq47}, we have finally obtained 
\begin{align}
m_{\mathrm{ADM}}&\leq\mathcal{C}+\limsup_{r\to +\infty}\frac{1}{2(n-2)(n-1)\vert \SSS^{n-1}\vert}\int\limits_{\partial B_{r}}\mathrm{R}_{\widetilde{g}_{0}}\,g(X,\nu_{\widetilde{g}_{0}})\,d\sigma_{\widetilde{g}_{0}}\,.\label{eq49}
\end{align}
We remark that the above inequality is true for any $\psi$ chart at infinity of $M$.
From now on we assume that $\psi$ satisfies condition~\eqref{eq58} regarding the decay rate of $\mathrm{R}_{\widetilde{g}_{0}}$ at $\infty$.
Since
\begin{align*}
\widetilde{g}_{0}(X,\nu_{\widetilde{g}_{0}})&=\widetilde{g}_{0;ij}X^{i}\nu_{\widetilde{g}_{0}}^{j}=\big(\,\delta_{ij}+O(\vert x\vert^{-p})\,\big)X^{i} (\nu_{\widetilde{g}_{0}}^{j}\pm\nu_{e}^{j})\\
&=g_{\R^{n}}(X,\nu_{e})+O(\vert x\vert^{-p+1})=\vert x\vert+O(\vert x\vert^{-p+1})\,,
\end{align*}
also using~\eqref{eq41app} we obtain
\begin{align}
\Big\vert\int\limits_{\partial B_{r}}\mathrm{R}_{\widetilde{g}_{0}}\,\widetilde{g}_{0}(X,\nu_{\widetilde{g}_{0}})\,d\sigma_{\widetilde{g}_{0}}\Big \vert\leq  C \int\limits_{\partial B_{r}}r^{-q}\big(\,r+O(r^{-p+1})\,\big) \,d\sigma_{e}\leq C r^{-q+n}\xrightarrow[r\to +\infty]{} 0 \,.
\end{align}
The fact that $m_{\mathrm{ADM}}\leq \mathcal{C}$ thus follows from~\eqref{eq49}.
All in all, the rigidity case $m_{\mathrm{ADM}}=\mathcal{C}$ of~\cite[Theorem 1.5]{hirsch} holds, which implies that $(M,g_{0})$ is the Schwarzschild manifold.
\end{proof}

\section{Appendix}\label{app}
In this Appendix we provide a proof of Lemma~\ref{firstitegralidenty} which is alternative and more self contained than the corresponding in~\cite{Virginia1}.
We underline that we will use Remark~\ref{remark2} widely.

\begin{proof}[\underline{Proof of Lemma~\ref{firstitegralidenty} (i)}.]
In $\mathring{M}\setminus\mathrm{Crit}(\varphi)$ and for every $\beta\geq 0$, we consider the smooth vector field 
\begin{align}
X_{\beta}:=\frac{\vert \nabla \varphi \vert^{\beta}_{g}\, \nabla \varphi}{\sinh \varphi}\,,
\end{align}
which is such that
\begin{align}\label{divX}
\mathrm{div}_{g}\,X_{\beta}&=\frac{\vert \nabla \varphi \vert^{\beta-2}_{g}\,\Big[\,\beta\,\nabla ^{2}\varphi(\nabla \varphi,\nabla \varphi)\,-\coth(\varphi)\vert \nabla \varphi \vert^{4}_{g} \Big]}{\sinh\varphi}\,
\end{align}
If $\{s\leq \varphi\leq S\}\cap \mathrm{Crit}(\varphi)=\emptyset$, then the statement is a straightforward application of the Divergence Theorem.
Now, suppose that $\{s\leq \varphi\leq S\}\cap \mathrm{Crit}(\varphi) \neq\emptyset$.
Since there always exists $\overline{s}\in (s,S)$ regular value of $\varphi$, up to splitting the right--hand side of~\eqref{partialfirstintegralidentities1} into two subintegrals, we can suppose without loss of generality that one among $s$ and $S$  is a regular value of $\varphi$. To fix the ideas, suppose that $S$ is the regular value.
We are going to change the function $\varphi$ in a neighbourhood of the set $\mathrm{Crit}(\varphi)$. 
To do this, for every $\varepsilon>0$ sufficiently small, applying Sard's Theorem to the smooth function $\varphi$, we can fix a positive real number $\delta(\varepsilon)$ such that $s+\delta(\varepsilon)<S$ is a regular value of $\varphi$ and $\delta(\varepsilon)<d\,\varepsilon$, where $d>0$ will be specified later.
Then, considering a smooth nonincreasing cut--off function $\zeta_{\varepsilon}:[0,+\infty)\to [0,1]$ satisfying the conditions
\begin{align}
\zeta_{\varepsilon}(\tau)=1\text{\quad in $\Big[\,0,\frac{1}{2}\varepsilon\,\Big]$}\,,\quad \vert \zeta_{\varepsilon}'(\tau)\vert \leq \frac c\varepsilon\text{\quad in $\Big[\,\frac{1}{2}\varepsilon,\frac{3}{2}\varepsilon\,\Big]$}\,,\quad \zeta_{\varepsilon}(\tau)=0
\text{\quad in $\Big[\,\frac{3}{2}\varepsilon,+\infty\,\Big)$}\,,\,\label{defcutoffzeta}
\end{align}
where $c$ is a positive real constant independent of $\varepsilon$, we define 
$$\varphi_{\varepsilon}:=\varphi-\zeta_{\varepsilon}(\vert \nabla \varphi\vert^{2}_{g})\,\delta(\varepsilon)\,.$$
Clearly,
\begin{align}\label{gradvareps}
\nabla \varphi_{\varepsilon}=\nabla \varphi-\delta(\varepsilon)\, \zeta'_{\varepsilon}\big( \,\vert \nabla \varphi \vert^{2}_{g}\,\big)\,\nabla\vert \nabla \varphi \vert^{2}_{g},
\end{align}
and
\begin{align}\label{remarkvarphivarepsilon}
\varphi=\varphi_{\varepsilon}\,\,\text{in}\,\, \Big\{\vert \nabla \varphi \vert^{2}_{g}\geq\frac{3}{2}\varepsilon\,\Big\}\,.
\end{align}
\noindent 
Note
that $s$ is a regular value for the function $\varphi_{\varepsilon}$.
To see this, let
$p$ be a point of $\{\varphi_{\varepsilon}=s\}$ and distinguish the two cases 
$$\vert \nabla \varphi \vert^{2}_{g}\,(p)\leq\frac{1}{2}\varepsilon\,;\quad \quad\quad \vert \nabla \varphi \vert^{2}_{g}\,(p)\overset{(\bigstar)}{>}\frac{1}{2}\varepsilon.$$
In the first case, $\zeta_{\varepsilon}(\vert \nabla \varphi\vert^{2}_{g})\equiv 1$ so that $s=\varphi_{\varepsilon}(p)= \varphi(p)-\delta(\varepsilon)$ and $\nabla \varphi_{\varepsilon}(p)=\nabla \varphi(p)$.
Since $s+\delta(\varepsilon)$ is a regular value for $\varphi$, $\nabla \varphi_{\varepsilon}(p)\neq 0$.
In the second case, observing that $s\leq\varphi(p)\leq s+\delta(\varepsilon)$ and therefore $p\in\{s\leq \varphi\leq S\}$, 
we have from~\eqref{gradvareps} that in $p$
\begin{align}
\vert \nabla \varphi_{\varepsilon}\vert_{g}&\geq \vert \nabla \varphi \vert_{g}-\delta(\varepsilon)\,\vert\zeta'_{\varepsilon}\vert\big( \,\vert \nabla \varphi \vert^{2}_{g}\,\big)\big\vert \nabla\vert \nabla \varphi \vert^{2}_{g}\big\vert_{g}
=\vert \nabla \varphi \vert_{g}\Big(1-2\delta(\varepsilon)\,\vert\zeta'_{\varepsilon}\vert\big( \,\vert \nabla \varphi \vert^{2}_{g}\,\big)\big\vert \nabla\vert \nabla \varphi \vert_{g}\big\vert_{g}\Big)\\
&\geq\vert \nabla \varphi \vert_{g}\Big(1-2d\,\varepsilon\, \frac{c}{\varepsilon}\, \max\limits_{\{s\leq \varphi\leq S\}}\big\vert\nabla\vert \nabla \varphi \vert_{g}\,\big\vert_{g}\,\Big)\,,
\end{align}
where $c$ is the constant appearing in~\eqref{defcutoffzeta}. Now, observe that $\max\limits_{\{s\leq \varphi\leq S\}}\big\vert\nabla\vert \nabla \varphi \vert_{g}\,\big\vert_{g}\,>0$, since otherwise, due to the presence of critical points in $\{s\leq \varphi\leq S\}$, there should be a connected component of $\{s\leq \varphi\leq S\}$ where $\nabla \varphi\equiv0$. But this is impossible because $\{s\leq \varphi\leq S\}=\overline{\{s<\varphi<S\}}$ (by Remark~\ref{disposizionelevelset}) and by the size of $\mathrm{Crit}(\varphi)$. Hence, choosing 
$$d\leq \frac{1}{4\,c\,\max\limits_{\{s\leq \varphi\leq S\}}\big\vert\nabla\vert \nabla \varphi \vert_{g}\,\big\vert_{g }}\,,$$ 
from above we obtain $\vert\nabla \varphi_{\varepsilon}\vert_{g}(p)\geq \frac{\vert\nabla \varphi\vert_{g}}{2}(p)$. In particular, from {\footnotesize $(\bigstar)$} we get that $\vert\nabla \varphi_{\varepsilon}\vert_{g}(p)>\frac{\varepsilon}{4}$.

\noindent Now, we apply 
the Divergence Theorem to the smooth vector field $\Xi_{4\varepsilon} X_{\beta}$ on $\{s<\varphi_{\varepsilon}<S\}$, where 
\[\Xi_{\varepsilon}\,:=\,1- \zeta_{\varepsilon}(\vert \nabla \varphi\vert^{2}_{g}).\]
Recalling that $U_{\mu}$ is defined as in~\eqref{eq62}, we obtain
\begin{align*}
\int\limits_{\{\varphi_{\varepsilon}=S\}}\,g\Big(\,\Xi_{4\varepsilon} X_{\beta},\frac{\nabla \varphi_{\varepsilon}}{\vert \nabla \varphi_{\varepsilon}\vert_{g}}\,\Big)d\sigma_{g}&-
\int\limits_{\{\varphi_{\varepsilon}=s\}}\,g\Big(\,\Xi_{4\varepsilon} X_{\beta},\frac{\nabla \varphi_{\varepsilon}}{\vert \nabla \varphi_{\varepsilon}\vert_{g}}\,\Big)\,d\sigma_{g}\\
&=\int\limits_{\{s<\varphi_{\varepsilon}<S\}}\Xi_{4\varepsilon}\frac{\vert \nabla \varphi \vert^{\beta-2}_{g}\,\Big[\,\beta\,\nabla ^{2}\varphi(\nabla \varphi,\nabla \varphi)\,-\coth(\varphi)\vert \nabla \varphi \vert^{4}_{g} \Big]}{\sinh\varphi}\,d\mu_{g}\\
&+2\int\limits_{ \big(U_{6\varepsilon} \setminus\overline{U_{2\varepsilon}}\big)\cap \{s<\varphi_{\varepsilon}<S\}}\,\frac{\,\chi'_{4\varepsilon}\big(\,\vert \nabla \varphi \vert^{2}_{g}\,\big)\,\vert \nabla \varphi \vert^{\beta}_{g}\,\nabla ^{2}\varphi(\nabla \varphi,\nabla \varphi)\,}{\sinh \varphi}\,d\mu_{g}\,.
\end{align*}
Note that $\{\varphi=S\}$ is compactly contained in $\{\vert \nabla \varphi\vert^{2}_{g}>\frac{3}{2}\,\varepsilon\}$ for every $\varepsilon$ sufficiently small, and $\Xi_{4\varepsilon}\equiv0$ in $\{\vert \nabla \varphi\vert^{2}_{g}\leq2\,\varepsilon\}\supset \{\vert \nabla \varphi\vert^{2}_{g}\leq\frac{3}{2}\,\varepsilon\}$. Then, by~\eqref{remarkvarphivarepsilon} we get
\begin{align}
\int\limits_{\{\varphi=S\}}\,\frac{\vert \nabla \varphi \vert^{\beta+1}_{g}}{\sinh\varphi}\,d\sigma_{g}&-
\int\limits_{\big\{\,\varphi=s,\vert \nabla \varphi \vert^{2}_{g}\geq\frac{3}{2}\varepsilon\,\big\}}\,\Xi_{4\varepsilon}\,\frac{\vert \nabla \varphi \vert^{\beta+1}_{g}}{\sinh\varphi}\,d\sigma_{g}\\
&=\int\limits_{\{s<\varphi<S\}}\Xi_{4\varepsilon}\,\frac{\vert \nabla \varphi \vert^{\beta-2}_{g}\,\Big[\,\beta\,\nabla ^{2}\varphi(\nabla \varphi,\nabla \varphi)\,-\coth(\varphi)\vert \nabla \varphi \vert^{4}_{g} \Big]}{\sinh\varphi}\,d\mu_{g}\\
&+2\int\limits_{ \big(U_{6\varepsilon} \setminus\overline{U_{2\varepsilon}}\big)\cap \{s<\varphi<S\}}\,\frac{\,\chi'_{4\varepsilon}\big(\,\vert \nabla \varphi \vert^{2}_{g}\,\big)\,\vert \nabla \varphi \vert^{\beta}_{g}\,\nabla ^{2}\varphi(\nabla \varphi,\nabla \varphi)\,}{\sinh \varphi}\,d\mu_{g}\,.\, \label{eq1}
\end{align}

\noindent
Looking at the left--hand side of~\eqref{eq1}, note that
\begin{align}
\Bigg\vert\int\limits_{ \big(U_{6\varepsilon}\setminus\overline{U_{2\varepsilon}}\big)\cap \{s<\varphi<S\}}\chi'_{4\varepsilon}\big(\,\vert \nabla \varphi \vert^{2}_{g}\,\big) \vert \nabla \varphi \vert^{\beta}_{g} \nabla ^{2}\varphi(\nabla \varphi,\nabla \varphi)\,d\mu_{g}\Bigg\vert
&\leq\frac{c}{4\,\varepsilon}\int\limits_{ U_{6\varepsilon}} \vert \nabla \varphi \vert^{\beta+2}_{g}\vert \nabla ^{2}\varphi\vert_{g}\,d\mu_{g}\\
&\leq C\frac{\varepsilon^{\frac{\beta}{2}+1}}{\varepsilon}\mu_{g}(U_{6\varepsilon} )\to 0\,,
\end{align}
\noindent where in the second inequality we have used Lemma~\ref{lemmabound} and the fact that $U_{\varepsilon}$ is contained in a compact set for every $\varepsilon<<1$ (which is a consequence of~\eqref{eq60}).
Moreover, by the Dominated Convergence Theorem, we have that
\begin{align}
\lim_{\varepsilon\to 0^+}\int\limits_{\{s<\varphi<S\}}\Xi_{4\varepsilon}\,&\frac{\vert \nabla \varphi \vert^{\beta-2}_{g}\,\Big[\,\beta\,\nabla ^{2}\varphi(\nabla \varphi,\nabla \varphi)\,-\coth(\varphi)\vert \nabla \varphi \vert^{4}_{g} \Big]}{\sinh\varphi}\,d\mu_{g}\\
&=\int\limits_{\{s<\varphi<S\}}\,\frac{\vert \nabla \varphi \vert^{\beta-2}_{g}\,\Big[\,\beta\,\nabla ^{2}\varphi(\nabla \varphi,\nabla \varphi)\,-\coth(\varphi)\vert \nabla \varphi \vert^{4}_{g} \Big]}{\sinh\varphi}\,d\mu_{g}\,.
\end{align}
Finally, note that $\{\zeta_{\varepsilon}\}$ can always be
chosen to be
nonincreasing in $\varepsilon$ so that, in turn, $\{\Xi_{\varepsilon}\}$ in nondecreasing. Therefore, looking at the left--hand side of~\eqref{eq1}, we have that
\begin{align}
\lim_{\varepsilon\to 0^+}\int\limits_{\big\{\,\varphi=s,\vert \nabla \varphi \vert^{2}_{g}\geq\frac{3}{2}\varepsilon\,\big\}}\,\Xi_{4\varepsilon}\,\frac{\vert \nabla \varphi \vert^{\beta+1}_{g}}{\sinh\varphi}\,d\sigma_{g}&=
\lim_{\varepsilon\to 0^+}\int\limits_{\{\,\varphi=s\}}\,\Xi_{4\varepsilon}\,\frac{\vert \nabla \varphi \vert^{\beta+1}_{g}}{\sinh\varphi}\,d\sigma_{g}\\
&=\int\limits_{\{\,\varphi=s \}}\,\frac{\vert \nabla \varphi \vert^{\beta+1}_{g}}{\sinh\varphi}\,d\sigma_{g}\,,
\end{align}
by the Monotone Convergence Theorem.
All in all, passing to the limit as $\varepsilon\to 0^+$ in~\eqref{eq1}, yields the desired identity.
\end{proof}

\begin{proof}[\underline{Proof of  Lemma~\ref{firstitegralidenty} (ii)}.]
Lemma~\ref{lemmabound} implies
$$\lim_{S\to+\infty}\int\limits_{\{\varphi=S\}}\frac{\vert \nabla \varphi \vert^{\beta+1}_{g}}{\sinh\varphi}\,d\sigma_{g}=0\,.$$
Note that
\begin{align}\label{eq2}
\frac{\vert \nabla \varphi \vert^{\beta-2}_{g}\,\Big[\,\coth(\varphi)\vert \nabla \varphi \vert^{4}_{g}-\beta\,\nabla ^{2}\varphi(\nabla \varphi,\nabla \varphi)\, \Big]}{\sinh\varphi}\in L^{1}\Big(\{\varphi\geq s\};\mu_{g}\Big)
\end{align}
because its absolute value belongs to $L^{1}_{\text{loc}}\big(\{\varphi\geq s\},\mu_{g}\big)$ immediately and to $L^{1}\big(\{\varphi\geq S\},\mu_{g}\big)$ for $S$ sufficiently big, applying the Coarea Formula coupled with~\eqref{eq60} and  Lemma~\ref{boundutili}.
Therefore, passing to the limit as $S\to +\infty$ in~\eqref{partialfirstintegralidentities1}
and using the Dominated Convergence Theorem gives the desired identity.\\
\end{proof}

\begin{proof}[\underline{Proof of Lemma~\ref{firstitegralidenty} (iii)}.]
Let $\beta\geq0$. We are assuming that the boundary $\partial M$ is a regular 
level set of $\varphi$ so that there exists $\epsilon>0$ such that 
$[0,\epsilon]\cap\mathrm{Crit}(\varphi)=\emptyset$. 
Therefore, applying the Divergence Theorem to the smooth vector field 
$\vert \nabla \varphi \vert^{\beta}_{g}\, \nabla \varphi$ in $\{0< \varphi<\epsilon\}$ 
yields
$$\Phi_{\beta}(\epsilon)-\Phi_{\beta}(0)=\int\limits_{\{0<\varphi<\epsilon\}}\,\beta\,\vert \nabla \varphi \vert^{\beta-2}_{g}\,\nabla ^{2}\varphi(\nabla \varphi,\nabla \varphi)\,d\mu_{g}\,.$$
In turn, the absolute continuity of the integral
implies the continuity of $\Phi_{\beta}$ at $0$.
By point
$(i)$ and again by the absolute continuity of the integral, we obtain
the right and the left continuity of the function 
\begin{equation}\label{Upsilonbeta}
\Upsilon_{\beta}:s\in(0,+\infty)\to\frac{\Phi_{\beta}(s)}{\sinh s}\in\R\,.
\end{equation}
Hence, $\Phi_{\beta}$ is continuous also in $(0,+\infty)$.
The integral representation of $\Phi_{\beta}$ follows directly from point $(ii)$. 
\end{proof}
\bigskip

\noindent
\textbf{\large{Acknowledgements.}} {\em 
The authors are members of the Gruppo Nazionale per l'Analisi Matematica, 
la Probabilit\`a
e le loro Applicazioni (GNAMPA), which is part of the Istituto Nazionale di 
Alta Matematica (INdAM), 
and they are partially funded by the GNAMPA project ``Aspetti geometrici in teoria del potenziale lineare e nonlineare''.
F. O. thanks L.~Benatti and C.~Mantegazza for 
useful discussions during the preparation of the manuscript.}

\bibliographystyle{amsplain}
\providecommand{\bysame}{\leavevmode\hbox to3em{\hrulefill}\thinspace}
\providecommand{\MR}{\relax\ifhmode\unskip\space\fi MR }
\providecommand{\MRhref}[2]{%
  \href{http://www.ams.org/mathscinet-getitem?mr=#1}{#2}
}
\providecommand{\href}[2]{#2}

\end{document}